\documentclass[11pt, draft]{amsart}
\usepackage{amssymb, amstext, amscd, amsmath, amssymb}
\usepackage{mathtools, xypic, paralist, color, dsfont, rotating}
\usepackage{verbatim}
\usepackage{enumerate,enumitem}
\usepackage{float}
\usepackage{setspace}
\usepackage{pifont}

\usepackage{tikz}
\usetikzlibrary{fit,positioning,arrows,automata,calc}
\tikzset{
  main/.style={circle, minimum size = 30pt, thick, draw = black!80, node distance = 10mm},
  connect/.style={-latex, thick},
  box/.style={rectangle, draw = white!100}
}
\usepackage{tikz-cd}

\usepackage{tikz}

\floatstyle{boxed} 
\restylefloat{figure}

\numberwithin{equation}{section}
\usepackage[a4paper]{geometry}
\usepackage{quoting}
\quotingsetup{vskip=.1in}
\quotingsetup{leftmargin=.6in} 
\quotingsetup{rightmargin=.6in}

\let\OLDthebibliography\thebibliography
\renewcommand\thebibliography[1]{
  \OLDthebibliography{#1}
  \setlength{\parskip}{0pt}
  \setlength{\itemsep}{2pt plus 0.5ex}
}

\makeatletter
\def\@cite#1#2{{\m@th\upshape\bfseries%
[{#1\if@tempswa{\m@th\upshape\mdseries, #2}\fi}]}}
\makeatother
\theoremstyle{plain}
\newtheorem{theorem}{Theorem}[section]
\newtheorem{corollary}[theorem]{Corollary}
\newtheorem{proposition}[theorem]{Proposition}
\newtheorem{lemma}[theorem]{Lemma}
\theoremstyle{definition}
\newtheorem{definition}[theorem]{Definition}
\newtheorem{example}[theorem]{Example}

\newtheorem{remark}[theorem]{Remark}

\newtheorem*{open}{Open Access Statement}
\theoremstyle{remark}

\mathtoolsset{centercolon}
  \newcommand{\A}{{\mathcal{A}}}
  \newcommand{\B}{{\mathcal{B}}}
  \newcommand{\C}{{\mathcal{C}}}
  \newcommand{\D}{{\mathcal{D}}}
  
  \newcommand{\F}{{\mathcal{F}}}
  \newcommand{\G}{{\mathcal{G}}}

  \newcommand{\K}{{\mathcal{K}}}

\renewcommand{\S}{{\mathcal{S}}}
  \newcommand{\T}{{\mathcal{T}}}

\newcommand{\eps}{\varepsilon}
\def\al{\alpha}
\def\be{\beta}

\def\de{\delta}

\def\io{\iota}
\def\ka{\kappa}
\def\la{\lambda}
\def\La{\Lambda}
\def\om{\omega}

\def\si{\sigma}

\newcommand\vthe{\vartheta}
\newcommand\vphi{\varphi}

\newcommand{\bC}{\mathbb{C}}

\newcommand{\bF}{\mathbb{F}}

\newcommand{\bN}{\mathbb{N}}
\newcommand{\bT}{\mathbb{T}}
\newcommand{\bZ}{\mathbb{Z}}
\newcommand{\bR}{\mathbb{R}}

\newcommand{\fB}{{\mathfrak{B}}}
\newcommand{\fC}{{\mathfrak{C}}}

\newcommand{\fK}{{\mathfrak{K}}}

\newcommand{\foral}{\text{ for all }}
\newcommand{\qand}{\quad\text{and}\quad}

\newcommand{\qfor}{\quad\text{for}\quad}

\newcommand{\ca}{\mathrm{C}^*}

\newcommand{\cenv}{\mathrm{C}^*_{\textup{env}}}

\newcommand{\ol}{\overline}
\newcommand{\wt}{\widetilde}
\newcommand{\wh}{\widehat}

\newcommand{\ad}{\operatorname{ad}}

\newcommand{\diag}{\operatorname{diag}}

\newcommand{\dist}{\operatorname{dist}}

\newcommand{\mt}{\emptyset}

\newcommand{\rank}{\operatorname{rank}}

\newcommand{\spn}{\operatorname{span}}

\newcommand{\sca}[1]{\left\langle#1\right\rangle} 
\newcommand{\nor}[1]{\left\Vert #1\right\Vert} 

\newcommand{\tes}[7]{
	\xymatrix@C=2cm@R=1.5cm{
		K_0\left(#1\right) \ar[r]^{#2} & K_0\left(#3\right) \ar[r]^{#4} & K_0\left(#5\right) \ar[d] \\
		K_1\left(#5\right) \ar[u] & K_1\left(#3\right) \ar[l]^{#7} & K_1\left(#1\right) \ar[l]^{#6}
	}
}

\addtocontents{toc}{\protect\setcounter{tocdepth}{1}}

\begin{document}

\title[Selfadjoint operator spaces]{On the embeddings of selfadjoint operator spaces}

\author[A. Chatzinikolaou]{Alexandros Chatzinikolaou}
\address{School of Electrical and Computer Engineering\\
National Technical University of Athens\\ Athens\\ 157 80\\ Greece}
\email{achatzinik@mail.ntua.gr}

\author[E.T.A. Kakariadis]{Evgenios T.A. Kakariadis}
\address{Department of Mathematics\\ National and Kapodistrian University of Athens\\ Athens\\ 1578 84\\ Greece}
\email{evkakariadis@math.uoa.gr}

\author[S.-J. Kim]{Se-Jin Kim}
\address{Department of Mathematics \\ KU Leuven\\ Celestijnenlaan 200B\\ Leuven \\ 3001\\ Belgium}
\email{sam.kim@kuleuven.be}

\author[I.A. Paraskevas]{Ioannis Apollon Paraskevas}
\address{Department of Mathematics\\ National and Kapodistrian University of Athens\\ Athens\\ 1578 84\\ Greece}
\email{ioparask@math.uoa.gr}



\begin{abstract}
We investigate when a map on a selfadjoint operator space $E$ is an embedding, i.e., when its unitisation in the sense of Werner is completely isometric.
Combining with results of Russell, of Ng, and of Dessi, the second and the last author, it is shown that this is equivalent to: (a) extending bounded positive functionals on each matrix level with the same norm; (b) extending quasistates to quasistates in each matrix level; (c) extending completely bounded completely positive maps with the same cb-norm; and (d) the map being a gauge maximal isometry in the sense of Russell.  If $E$ is approximately positively generated and $\mathrm{C}^*(E)$ is unital, or if $E_{sa}$ is singly generated, then completely positive maps on $E \subseteq \mathcal{B}(H)$ have completely positive extensions on $\mathrm{C}^*(E)$, but possibly not with the same cb-norm; and this is not enough for the inclusion $E \subseteq \mathrm{C}^*(E)$ to be an embedding. 
We show that the inclusion $E \subseteq \mathrm{C}^*(E)$ is always an embedding when $E$ is completely approximately 1-generated, and we fully resolve the case when $E_{sa}$ is singly generated. Combining with the works of Salomon, Humeniuk--Kennedy--Manor, and previous work of the third author, we show that if the inclusion $E \subseteq \mathrm{C}^*(E)$ is an embedding, then rigidity at zero, in the sense of Salomon, coincides with $E$ being approximately positively generated. 
Consequently, we show that $E$ is approximately positively generated if and only if $M_n(E)$ is approximately positively generated for all $n\in \mathbb{N}$, thus extending a previous result of Humeniuk--Kennedy--Manor to the approximation setting. 
As an application we show that hyperrigidity of $E$ in $\mathrm{C}^*(E)$ allows to identify $\mathrm{C}^*(E)$ as the C*-envelope of $E$ in several (non-unital) contexts.
\end{abstract}

\maketitle

\section{Introduction}

\subsection{Framework}

Operator systems, i.e., closed, selfadjoint subspaces of $\B(H)$ containing its unit, have a central role in the theory of Operator Algebras.
Lately, the community has been actively considering selfadjoint operator subspaces, but which need not be unital \cite{CS21, CS22, DKP25, EKT25, HKM23, JN25, KKM23, Kim24, Ng22, Rus23}. 
Those appear naturally, for example when comparing operator systems up to their stabilisations as in the work of Connes and van Suijlekom \cite{CS21}.
The absence of a unit disrupts the link between the norm and the matrix order structure, and thus many of the standard results from the unital setting no longer apply. 
This creates significant difficulties in obtaining even fundamental theorems like the celebrated Arveson's Extension Theorem. 

Matrix ordered operator spaces arise naturally in the study of duality. 
While the dual of a finite dimensional operator system is again an operator system, in the infinite dimensional case the dual space generally fails to be unital, and thus naturally falls into the broader category of matrix ordered operator spaces. 
Connes and van Suijlekom \cite{CS21,CS22} link aspects of noncommutative geometry to this framework through the study of spectral truncations and tolerance relations. 
Selfadjoint operator spaces appear naturally in the construction of symmetrisations by Eleftherakis, the second author and Todorov \cite{EKT25} with applications to Morita theory. 

In his seminal work, Werner \cite[Lemma 4.8 and Theorem 4.15]{Wer02} characterises precisely the compatibility between the norm and the order structure in a matrix ordered operator space so that it is concretely realised in some $\B(H)$ by a completely isometric complete order embedding. 
A \emph{selfadjoint operator space} is exactly a matrix ordered operator space that embeds in this way, so that both its matrix norm structure and its matrix cone structure are inherited from $\B(H)$. 
Towards this end, Werner \cite{Wer02} introduces the notion of the unitisation $E^\#$ of a matrix ordered operator space $E$, with the property that the (unique) unital extension of a completely contractive completely positive map from $E$ to $E^\#$ is completely positive.
Werner's unitisation $E^\#$ is an operator system, and the matrix ordered operator space $E$ is a (concrete) selfadjoint operator space if and only if the canonical inclusion $E \subseteq E^\#$ is completely isometric.

Positivity in Werner's unitisation is extrinsic, in the sense that it is described in terms of the states on every matrix level. 
On the other hand van Dobben de Bruyn in \cite[Theorem 8.3]{vD25} provides an intrinsic characterisation of positivity in relation to the distance from the positive cones. 
The distance of a selfadjoint element from the positive cone plays also a central role in the theory of maximal gauges of Russell \cite{Rus23}. 

\subsection{The embedding problem}

Although any completely isometric complete order embedding has a unique completely positive extension on the unitisation, this extension is not always completely isometric.
The maps that do have this property are called \emph{embeddings}.
Russell \cite[Example 4.1]{Rus23}, and Kennedy, the third author and Manor \cite[Example 2.14]{KKM23} show that this fails even for (finite dimensional) singly generated spaces (in commutative C*-algebras).
Moreover, Connes and van Suijlekom \cite[Remark 2.27]{CS21} show that completely isometric complete order maps are not the suitable morphisms to obtain a meaningful extremal theory in this category, and one needs to restrict to the class of embeddings.
Let us recall here this example to illustrate the subtlety of the question.

\begin{example} \label{E:intro} \cite[Remark 2.27]{CS21}, \cite[Example 2.14]{KKM23},  \cite[Example 4.1]{Rus23}.
For $r\in(0,1]$, consider
\[
x_r := \begin{bmatrix} 1 & 0 \\ 0 & -r \end{bmatrix}.
\]
Set $E_r:=\spn\{x_r\} \subseteq \bC^2 \subseteq M_2(\bC)$, which is a selfadjoint operator space with trivial cones.
Because of this triviality, the map
\[
\vphi \colon E_r \to \bC; \vphi(x_r) = -1,
\]
is a (completely) contractive (completely) positive functional.
When $r < 1$ this map fails to extend to a (completely) contractive (completely) positive functional on $\bC^2$, and thus to $M_2(\bC)$ as $-1 \notin \ol{{\rm conv}}\{ \si_{M_2(\bC)}(x_r)\}$.
Moreover, the map
\[
\vthe \colon E_r \to M_2(\bC); \vthe(x_r) = -x_r,
\]
is a completely isometric complete order embedding.
When $r < 1$ this map fails to extend to a complete isometry on the unitisations, as in this case the elements $\pm x_r$ would generate an element in the C*-envelope of $E_r$ with an empty spectrum.

One may suspect that it is the triviality of the cones that enforces this
pathology.
However this is not true as, when $r = 1$, the inclusion $E_1 \subseteq
M_2(\bC)$ is an embedding.
\end{example}

This example exhibits that embeddings should reflect both the extrinsic and the intrinsic aspects of positivity inherited from the C*-cover, e.g., extensions of maps, spectrum, positive decomposition and distance from cones.

In terms of extensions, Werner \cite{Wer04} shows that a completely isometric complete order embedding $\vthe \colon E \to F$ is an embedding if and only if any completely contractive completely positive map $\phi \colon M_n(E) \to M_k(\bC)$ has a completely contractive completely positive extension $\wt{\phi} \colon M_n(F) \to M_k(\bC)$ with $\wt{\phi} \circ \vthe = \phi$, for all $k, n \in \bN$.
There are two refinements analogous to the operator system theory.
Ng \cite{Ng11} (in a broader context) shows that $\vthe$ is an embedding if and only if any completely contractive completely positive map $\phi \colon E \to M_k(\bC)$ has a completely contractive completely positive extension $\wt{\phi} \colon F \to M_k(\bC)$ with $\wt{\phi} \circ \vthe = \phi$, for all $k \in \bN$.
Dessi with the second and last authors \cite{DKP25} show that $\vthe$ is an embedding if and only if any contractive positive functional on $\vphi \colon M_n(E) \to \bC$ has a contractive positive extension $\wt{\vphi} \colon M_n(F) \to \bC$ with $\wt{\vphi} \circ \vthe = \vphi$, for all $n \in \bN$.

In terms of intrinsic properties, the association with distance formulas is depicted by Russell \cite{Rus23}.
Therein, a completely isometric complete order embedding $\vthe \colon E\to F$ is called a \textit{gauge maximal isometry} if
\[
\dist(x,M_n(E)_+) =  \dist(\vthe^{(n)}(x),M_n(F)_+)
\foral
x\in M_n(E)_{sa}.
\]
The motivation in \cite{Rus23} has been to pursue an abstract characterisation of selfadjoint operator spaces from the given data as in \cite{Wer02}.
As shown in \cite{Rus23} not all completely isometric complete order embeddings are gauge maximal isometries, leaving open the question of the complete characterisation.
Our first main theorem below combines the aforementioned works to the following characterisation of embeddings. 

\medskip

\noindent
{\bf Theorem A.} (\cite{Ng11}, \cite{Rus23}, \cite{DKP25}, Theorem \ref{T:equivemb}, Theorem \ref{T:arvemb})
{\it Let $\vthe \colon E\to F$ be a completely isometric complete order embedding between selfadjoint operator spaces.  
Then the following are equivalent:
\begin{enumerate}
\item The map $\vthe \colon E \to F$ is an embedding.
\item Every map $\vphi \in {\rm CCP}(M_n(E), \bC)$ extends to a map $\wt{\vphi} \in {\rm CCP}(M_n(F), \bC)$ with $\wt{\vphi} \circ \vthe = \vphi$, for every $n\in \bN$.
\item Every map $\phi \in {\rm CCP}( E, M_k(\bC))$ extends to a map $\wt{\phi}\in {\rm CCP}( F, M_k(\bC))$  with $\wt{\phi} \circ \vthe = \phi$, for every $k\in \bN$.
\item Every map $\phi\in {\rm CCP}( E, \B(K))$ extends to a map $\wt{\phi}\in {\rm CCP}(F , \B(K))$ with $\wt{\phi} \circ \vthe = \phi$, for every Hilbert space $K$.
\item The map $\vthe \colon E \to F$ is a gauge maximal isometry.
\end{enumerate}
}

\medskip

Equivalence [(i) $\Leftrightarrow$ (ii)] is shown in \cite[Corollary 2.3]{DKP25}.
Equivalence [(i) $\Leftrightarrow$ (iii)] is shown in a broader context in \cite[Lemma 4.2 (a)]{Ng11}.
Equivalence [(iv) $\Leftrightarrow$ (v)] is shown in \cite[Corollary 4.2]{Rus23}.
Equivalence [(iii) $\Leftrightarrow$ (iv)] follows by a BW-compactness argument and is provided in Theorem \ref{T:arvemb}.
Furthermore, equivalence [(i) $\Leftrightarrow$ (v)] is directly shown in Theorem \ref{T:equivemb} with a connection of Russell's gauge maximality with the description of the positive cones of the unitisation by van Dobben de Bruyn \cite{vD25}. 
A key step in this approach is the distance formula
\[
\dist(x, M_n(E)_{+}) = \sup_{\vphi \in {\rm CCP}(M_n(E), \bC)} -\vphi(x)
\foral 
x \in M_n(E)_{sa}, n \in \bN,
\]
that we show in Theorem \ref{T:stateformulaE}.
In passing, in Corollary \ref{C:spec} we show that an inclusion $E \subseteq \ca(E)$ is an embedding if and only if
\[
\dist(x, M_n(E)_+) = \max_{\la \in \si_{M_n(\C)}(x)} -\la
\foral
x \in M_n(E)_{sa}, n \in \bN.
\]
Item (iv) above is a version of Arveson's Extension Theorem, an important tool for any further developments in the theory.

\subsection{The automatic embedding property}

We next study when an inclusion $E \subseteq \B(H)$ (or equivalently, $E \subseteq \ca(E)$) is an embedding. 
As test cases we have that $E \subseteq \ca(E)$ is an embedding when $E$ contains an approximate unit in a C*-subalgebra $\A \subseteq E$ as in \cite{DKP25}, or when $E$ contains a matrix norm-defining approximate order unit for $\ca(E)$ as in \cite{CS22} (it is not known if the first class is contained in the second).
We focus on two further classes that have tractable cone structure: when $E$ is \emph{approximately positively generated}, i.e., $E_{sa} = \ol{E_{+} - E_{+}}$, or when $E_{sa}$ is singly generated. 
These cases are of particular interest as completely positive maps may attain completely positive extensions.

\medskip

\noindent 
{\bf Theorem B.} (Corollary \ref{C:exten}, Proposition \ref{P:1dposext}) 
{\it Let $E \subseteq \ca(E)$ be a selfadjoint operator space.
Then the following hold:
\begin{enumerate}
\item If $E$ is approximately positively generated and $\ca(E)$ is unital, then every completely bounded completely positive map $\phi\colon E \to \B(K)$ has an extension to a (completely bounded) completely positive map $\wt{\phi} \colon\ca(E) \to \B(K)$.
\item If $E = \spn\{x\}$ for some $x \in E_{sa}$, then every (completely bounded)  completely positive map $\phi\colon E \to \B(K)$ has an extension to a (completely bounded) completely positive map $\wt{\phi} \colon\ca(E) \to \B(K)$.
\end{enumerate}
}

\medskip

Approximately positively generated spaces $E$ are in abundance and this extra structure demonstrates strong connections with the ambient C*-algebra.
In Proposition \ref{P:complappos} we combine the notion of rigidity at zero of Salomon \cite{Sal17} with the works of Humeniuk--Kennedy--Manor \cite{HKM23} and of the third author \cite{Kim24} to show that $E$ is approximately positively generated if and only if $M_n(E)$ is approximately  positively generated for all $n\in \bN$.
This is an extension of \cite[Proposition 8.4]{HKM23} to the approximation setting.

The proof of item (i) in Theorem B relies on Lemma \ref{L:extendable_closed} which may be of independent interest.
Therein it is shown that the set of extendable completely bounded and completely positive maps is closed in the point-weak* topology.
However this extra feature is not enough for inclusions to be embeddings.
In Example \ref{E:nonemb} we demonstrate that a positive extension of a positive functional may fail to preserve the norm, motivating the consideration of approximate positive generation together with norm control over the positive decompositions of selfadjoint elements.

Such decompositions have been extensively studied in \cite{HKM23, JN25, KV97, Ng22}.
A selfadjoint operator space $E$ is said to be \emph{completely $\ka$-generated} if there exists $\ka \geq 1$ such that
\[
(M_n(E)_{sa})_1 \subseteq (M_n(E)_+)_{\ka} - (M_n(E)_+)_{\ka},
\]
where $(M_n(E)_{sa})_1$ denotes the unit ball of selfadjoint elements in $M_n(E)$, and $(M_n(E)_+)_{\ka}$ denotes the ball of radius $\ka$ in the positive cone. 
Following their terminology, $E$ is said to be \emph{completely approximately $\ka$-generated} if there exists $\ka \geq 1$ such that
\[
(M_n(E)_{sa})_1 \subseteq \ol{(M_n(E)_+)_{\ka} - (M_n(E)_+)_{\ka}}.
\]

Ng proves in \cite[Theorem 3.9]{Ng22} that a selfadjoint operator space $E$ is completely $\ka$-generated for some $\ka>0$ if and only if it is completely approximately $\ka'$-generated for some $\ka'>0$ if and only if it is \emph{dualisable}, i.e., $E^d$ is a selfadjoint operator space under the canonical dual matrix cone and with a matrix norm equivalent to the cb-norm. 
Humeniuk--Kennedy--Manor show in \cite[Theorem 5.8]{HKM23} that dualisability is equivalent to complete $\ka''$-normality of the dual space. 
Moreover, in \cite[Theorem 5.8]{HKM23} they provide an intrinsic characterisation of dualisability in terms of the corresponding pointed compact nc-convex set of the selfadjoint operator space, see \cite[Theorem 4.5]{KKM23} of Kennedy, the third author and Manor.
We show that when $E$ is completely approximately 1-generated then it satisfies the automatic embedding property in the sense of Definition \ref{D:aep}.
Thus, apart from a good duality framework, in completely approximately 1-generated selfadjoint operator spaces one obtains Arveson's Extension Theorem. 

\medskip

\noindent
{\bf Theorem C.} (Theorem \ref{T:1apppos})
{\it Let $E$ be a completely approximately 1-generated selfadjoint operator space. 
Then $E$ has the automatic embedding property.}

\medskip

We note that due to the Hahn--Banach Theorem, every positive functional admits a selfadjoint extension with the same norm.
Complete approximate 1-generation imposes that such extensions are automatically positive due to a Jordan decomposition argument.
With Example \ref{E:nonemb} we stress that $E$ may be positively generated, but not completely 1-generated, even if $\ca(E)$ is unital (we further illustrate the independence of approximate positive generation, unitality of the generated C*-algebra, and embeddings in Example \ref{E:tables}).

Let us return now to Example \ref{E:intro} from \cite{CS21, KKM23, Rus23} on singly generated selfadjoint operator spaces, and explain where the problem lies.

\medskip

\noindent
{\bf Theorem D.} (Corollary \ref{C:1dproper}, Theorem \ref{T:1d})
\emph{Let $x\in \B(H)_{sa}$ and let $x_{\pm}\in \B(H)_+$ be the unique positive and negative components of $x$. Set $E := \spn\{x\}$.
Then we have the following cases:\\
\textup{(a)} If $x_{+} = 0$ or $x_{-} = 0$, then the inclusion $\spn\{x\} \subseteq \ca(E)$ is an embedding.\\
\textup{(b)} If $x_{\pm} \neq 0$, then the following are equivalent:
\begin{enumerate}
\item The inclusion of $\spn\{x\}$ inside $\ca(E)$ is an embedding.
\item $\|x_{+}\| = \|x_{-}\|$.
\item The map $\phi \colon \spn\{x\} \to \spn\{x\}$ with $\phi(x) = -x$ has a completely contractive completely positive extension $\wt{\phi} \colon \ca(E) \to \B(H)$.
\item The functional $\vphi \colon \spn\{x\} \to \bC$ with $\vphi(x) = -\|x\|$ has a (completely) contractive (completely) positive extension $\wt{\vphi} \colon \ca(E) \to \bC$.
\end{enumerate}}

\medskip

Therefore the extension of the maps $\phi$ and $\vphi$ are --each on its own-- the only obstruction for the embedding property in Example \ref{E:intro}.
Likewise this shows why the inclusion $E_1 \subseteq \bC^2$ is an embedding.
In this case the map $\vphi$ is the restriction of the compression to the (2,2) entry, while $\phi$ is implemented by the unitary permutation of the basis elements.

Finally we give applications of the embedding property with respect to extremal theory.
In general we have three classes to investigate:
\begin{align*}
\F_1 & := \{\phi \colon E\to \B(K) : \text{$\phi$ completely isometric completely positive} \}, \\
\F_2 & := \{\vthe \colon E\to \B(K) : \text{$\vthe$ completely isometric complete order embedding} \}, \\
\F_3 &:= \{\vthe \colon E \to \B(K) : \text{$\vthe$ embedding} \}.
\end{align*}
We denote by $\partial \F_i$ the terminal object in each category $\F_i$, $i=1,2,3$ if it exists.
Although $\partial \F_3$ always exists by \cite[Theorem 2.25]{CS21} it is not generally true that $\partial \F_1$ or $\partial \F_2$ exist, or that they are the same if they exist.
The example at the beginning of the introduction gives a case where $\partial \F_1$ and $\partial \F_2$ do not exist.
Blecher--Kirkpatrick--Neal--Werner show in \cite[Theorem 2.3]{BKNW07} that $\partial \F_1$ does exist for $E$ being an approximately positively generated selfadjoint operator space.
Under the automatic embedding property, in Corollary \ref{C:F_2} we show that $\partial \F_2 = \partial \F_3$ when $E$ is completely approximately 1-generated. 
Moreover, we obtain the following result regarding hyperrigidity.

\medskip

\noindent 
{\bf Theorem E.} (Proposition \ref{P:embd_plus_hyper}, Corollary \ref{C:emb_apr_hyper}) 
{\it Let $E \subseteq \ca(E)$ be a hyperrigid selfadjoint operator space.
Then $\cenv(E)\cong \ca(E)$ by a canonical $*$-isomorphism if at least one of the following conditions is satisfied:
\begin{enumerate}
\item $E$ is separating for $\ca(E)$ and $E\subseteq \ca(E)$ is an embedding.
\item $\ca(E)$ is unital and $E\subseteq \ca(E)$ is an embedding.
\end{enumerate}

In particular,  $\cenv(E)\cong \ca(E)$ by a canonical $*$-isomorphism when:
\begin{enumerate}
\item[\textup{(a)}] $\A \subseteq E$ is a C*-algebra such that $\A \cdot E \subseteq E$ and there is an approximate unit $(e_\la)_\la \subseteq \A$ such that $\lim_\la e_\la x = x$ for all $x \in E$.
\item[\textup{(b)}] $E$ contains a matrix norm-defining approximate order unit $(e_\la)_\la$ for $E$.
\item[\textup{(c)}] $E$ is completely approximately 1-generated.
\end{enumerate}
}

\medskip

The hyperrigidity property of a generating set in a C*-algebra was introduced by Arveson \cite{Arv11} where it was also shown that it implies minimality of the C*-cover when the generating set is a unital operator system.
Salomon \cite{Sal17} introduced the rigidity at zero which now has become a central element in the theory.
Since \cite{Arv11} there has been a growing interest in researching the property in specific classes (perhaps too many for an inclusive list here).
This line of research has regained interest with the refutation of Arveson's Hyperrigidity Conjecture for unital noncommutative operator algebras by Bilich--Dor-On \cite{BD24}.
In a very recent preprint Clou{\^a}tre \cite{Clo25} identifies the obstruction(s) to that problem, while the commutative case remains open.

\subsection{Contents}

The structure of the manuscript is as follows.
In Section \ref{S:prelim} we gather the preliminaries on selfadjoint operator spaces, their unitisations, and their C*-extensions.

In Section \ref{S:emb} we give the complete characterisation for a completely isometric complete order embedding between selfadjoint operator spaces to be an embedding. 
En route we give another proof of Werner's result \cite{Wer02} that the C*-algebraic unitisation coincides as an operator system with the unitisation as a selfadjoint operator space. 
We show that for a selfadjoint operator space endowed with the maximal gauge, Werner's unitisation is isomorphic to Russell's unitisation.

In Section \ref{S:pos gen} we focus on approximately positively generated spaces.
We show that if the generated C*-algebra is unital, then we can extend completely bounded completely positive maps to completely bounded completely positive maps. 
This is not equivalent to the inclusion being an embedding as exhibited in Example \ref{E:nonemb}.
We show that completely approximately 1-generated selfadjoint operator spaces admit an automatic embedding property.

In Section \ref{S:1d} we focus on singly generated selfadjoint operator spaces.
We show that extensions of completely positive maps always exist, but not always with the same cb-norm.
We give equivalent conditions for an inclusion $\spn\{x\} \subseteq \B(H)$ with $x \in \B(H)_{sa}$ to be an embedding.

In Section \ref{S:ext} we give some applications of the characterisation of embeddings in terms of co-universal objects and hyperrigidity.

\section{Preliminaries} \label{S:prelim}

We present some elements of the theory of selfadjoint operator spaces in the sense of \cite{CS21, CS22, DKP25, KKM23, Rus23, Wer02, Wer04}; see also \cite{BL04, Pau02} for operator systems theory.
A \emph{matrix ordered operator space} $E$ is a norm-closed operator space with a completely isometric involution and a family of norm-closed cones $C_n\subseteq M_n(E)_{sa}$ such that:
\begin{enumerate}
\item for every $n \in \bN$, the cone $C_n$ is proper in $M_n(E)_{sa}$;
\item for every $n, m \in \bN$, we have $a x a^*\in C_m$ for every $a \in M_{m,n}$ and $x\in C_n$.
\end{enumerate}
We will denote the positive cone of $M_n(E)$ by $M_n(E)_+$.
A map $\phi \colon E \to F$ between matrix ordered operator spaces $E$ and $F$ is called \emph{positive} if it is selfadjoint, i.e., if $\phi^*=\phi$ where $\phi^*(x):=\phi(x^*)^*$ for every $x\in E$, and $\phi(E_+)\subseteq F_+$. 
Note that $\phi$ being selfadjoint  does not follow from the positivity condition as the cones do not necessarily span the spaces. 
We say that a map $\phi \colon E \to F$ is \emph{completely positive} if for every $n\in \bN$ the amplification 
\[
\phi^{(n)} \colon M_n(E) \to M_n(F); [x_{ij}] \mapsto [\phi(x_{ij})],
\]
is positive.
For $E$ and $F$ matrix ordered operator spaces we write
\begin{align*}
{\rm CB}(E, F) &:= \{\phi \colon E \to F : \phi \text{ is completely bounded} \},\\
{\rm CCP}(E, F) &:= \{\phi \colon E \to F : \phi \text{ is completely contractive and completely positive}\}.
\end{align*}
A map $\phi \colon E \to F$ is a \emph{complete order embedding} if it is completely positive and invertible onto its image with a completely positive inverse. If $\phi$ is also surjective it is called a \emph{complete order isomorphism}.

A matrix ordered operator space $E$ is called \emph{completely $\ka$-normal} if there exists a $\ka>0$ such that for every $n\in \bN$ and every $x\in M_n(E)_{sa}$ we have 
\[
\|x\|\leq \ka \max\{ \|u\|, \|v\|\} 
\text{ whenever } 
u,v\in M_n(E)_{sa}
\text{ satisfy } u\leq x \leq v.
\] 
In \cite[Remark 4.14]{Wer02} it is stated that a matrix ordered operator space $E$ is completely $\ka$-normal if and only if there exists a Hilbert space $H$ and a complete order embedding $\phi \colon E \to \B(H)$ such that 
\[
\|\phi \colon E \to F\|_{\rm cb} \leq 1
\qand 
\|\phi^{-1} \colon \phi(E) \to F\|_{\rm cb}\leq \ka.
\]
For $\ka=1$ this is shown in \cite[Corollary 4.1]{Rus23}, and for general $\ka>0$ in \cite[Corollary 8.13]{vD25}.
Here we will restrict to completely $1$-normal matrix ordered operator spaces, and thus to the concrete case; that is, by a \emph{selfadjoint operator space} we will mean a selfadjoint norm-closed subspace of some $\B(H)$ with the induced matrix cones and matrix norms via the inclusion.

For a selfadjoint operator space $E\subseteq \B(H)$ we write $\ca(E)$ for the C*-algebra it generates inside $\B(H)$. 
We note that the same proof as in \cite[Proposition 3.8]{Pau02} yields that contractive (positive) functionals of a selfadjoint operator space are automatically completely contractive (resp. completely positive). 
Henceforth, we will not use the term completely when we refer to positivity or contractivity of functionals.

In \cite[Definition 4.7]{Wer02} Werner introduces the following notion of unitisation for a matrix ordered operator space $E$.
Consider the direct sum linear space  $E^{\#} := E \oplus \bC$ with an involution given by $(x, a)^*:=(x^*,\ol{a})$.
Declare a selfadjoint element $(x,a) \in M_n(E^{\#})$ to be \emph{positive} if
\[
a \geq 0 \text{ and } \vphi(a_\eps^{-1/2} x a_\eps^{-1/2}) \geq -1 \foral \eps>0 \text{ and } \vphi \in {\rm CCP}(M_n(E), \bC),
\]
where $a_\eps := a + \eps I_{M_n}$.
We will use the following equivalent definition
\[
(x, a) \in M_{n}(E^{\#})_{+} 
\; \Longleftrightarrow \; 
a \geq 0 
\; \; \textup{ and } 
\sup_{\vphi \in {\rm CCP}(M_n(E), \bC)} -\vphi(a_{\eps}^{-1/2} x a_{\eps}^{-1/2}) \leq 1 \foral \eps>0.
\]

In \cite[Lemma 4.8]{Wer02} Werner shows that $E^{\#}$ is an operator system, and that $E$ is a selfadjoint operator space if and only if the inclusion map 
\[
\io_E \colon E \hookrightarrow E^{\#}; x \mapsto (x,0),
\]
is a completely isometric complete order embedding. 
In what follows we will omit the explicit reference to the map $\io_E$ and identify $E$ with its copy inside $E^\#$ whenever $E$ is a selfadjoint operator space. 
In \cite[Lemma 4.9]{Wer02} it is shown that whenever $E$ and $F$ are matrix ordered operator spaces and $\phi \colon E \to F$ is a completely contractive completely positive map, then the natural unitisation 
\[
\phi^\# \colon E^\#\to F^\#; (x, a) \mapsto (\phi(x), a),
\]
is unital and completely positive.
If in addition $\phi$ is a completely isometric complete order isomorphism, then $\phi^\#$ is a unital complete order isomorphism \cite[Corollary 2.14]{CS21}.

However, $\phi^\#$ is not always a unital complete order embedding if $\phi$ is assumed to be merely a completely isometric complete order embedding.
Following \cite[Section 2]{KKM23}, a map $\vthe \colon E \to F$ is called an \emph{embedding} if it is a completely isometric complete order embedding, and the unitisation map $\vthe^\#$ is a unital complete order embedding of operator systems; equivalently $\vthe^{\#}$ is completely isometric.
In \cite[Corollary 4.9]{Wer02} Werner shows that the unitisation construction recovers the known unitisation when restricted to C*-algebras. 
We will give an alternative proof to this fact in Corollary \ref{C:C*-unit} through our characterisation of embeddings in Theorem \ref{T:equivemb}.

Injective $*$-homomorphisms between C*-algebras are embeddings, as the unitisation of an injective $*$-homomorphism is injective, and thus completely isometric.
Hence for $E \subseteq\B(H)$ the inclusion map is an embedding if and only if the inclusion map $E \subseteq \C$ is an embedding for any C*-algebra $E\subseteq \C \subseteq \B(H)$.
In \cite[Example 2.14]{KKM23} and \cite[Example 4.1]{Rus23}, it is shown that even such inclusions are not automatically embeddings; see also \cite[Example 2.4]{DKP25}.
By seeing $\B(H)$ as the (1,1) corner of $\B(H \oplus \bC)$, then $\ca(E)^\#$ is $*$-isomorphic to $\ca(E) + \bC I_{H \oplus \bC}$.
Hence $E \subseteq \B(H)$ is an embedding if and only if for every $n \in \bN$ and $(x,a) \in M_n(E^\#)$ we have
\[
\|(x,a)\|_{\#}^{(n)} = \|x +a I_{H \oplus \bC}\|,
\]
where $\nor{\cdot}_{\#}^{(n)}$ is the norm induced by the matrix order unit $(0,1)$ in $E^\#$.
Note that \lq\lq$\geq$'' always holds by the construction of $E^\#$.

We have the following proposition about the concrete realisation of the unitisation of a selfadjoint operator space that embeds in the C*-algebra it generates.

\begin{proposition}\label{P:concreal}
Let $E \subseteq \B(H)$  be an embedding of a selfadjoint operator space. Then the following hold:
\begin{enumerate}
\item If $\ca(E)\subseteq \B(H)$ is not unital, or if it is unital and $1_{\ca(E)}\neq I_{H}$, then $ E^{\#} \cong E+\bC I_{H}$.
\item If $ \ca(E) \subseteq \B(H)$ is unital and $1_{\ca(E)}=I_{H} $, then $E^{\#} \cong E + \bC I_{H\oplus \bC}$ with the canonical identification $E \subseteq \B(H) \subseteq \B(H \oplus \bC)$.
\end{enumerate}
\end{proposition}

\begin{proof}
Recall that, if $\C \subseteq \B(H)$ is a C*-algebra, then we have the following cases for identifying $\C^\#$.
If $\C$ is not unital, or if $\C$ is unital and $1_{\C} \neq I_H$, then $\C^\# \cong \C + \bC I_H \subseteq \B(H)$ by the map
\[
\Phi \colon \C^\# \to \B(H); (x, a) \mapsto x + a I_H.
\]
If $\C$ is unital and $1_\C = I_H$, then $\C^\# \cong \C \oplus \bC \subseteq \B(H \oplus \bC)$ by the map
\[
\Phi \colon \C^\# \to \B(H \oplus \bC); (x, a) \mapsto (x + a I_H, a).
\]
Now the proof follows since the embedding property allows to write $E^\#\subseteq \ca(E)^\#$.
\end{proof}

\begin{remark}\label{R:covers}
Let $E$ be a selfadjoint operator space. An adaptation of the proof of \cite[Theorem 2.25 (i)]{CS21} yields that if $ \psi \colon E^{\#}\to \B(H) $ is a unital complete order embedding, then $\psi|_{E}$ is an embedding.
\end{remark}

The notion of embeddings is sufficient for an extremal theory.
In \cite[Theorem 2.25]{CS21} Connes and van Suijlekom prove that there exists a (necessarily unique up to $*$-isomorphisms) C*-envelope in the category of selfadjoint operator spaces with respect to embeddings, rather than just completely isometric complete order embeddings; see also \cite[Proposition 3.5]{Wer04} and \cite[Proposition and Definition 5.11]{Ng11}.
Let us fix notation.

A \emph{C*-cover} of a selfadjoint operator space $E$ is a pair $(\C, \vthe)$ where $\C$ is a C*-algebra and $\vthe \colon E \to \C$ is an embedding such that $\ca(\vthe(E)) = \C$.
In \cite[Theorem 2.25]{CS21} it is shown that any such $E$ admits a unique minimal C*-cover $\cenv(E)$, which is called \emph{the C*-envelope of $E$}.
In particular, it is shown that $\cenv(E)$ is the C*-algebra generated by $E$ inside $\cenv(E^{\#})$.
In \cite[Theorem 2.25]{CS21} it is shown that, if $E$ is an operator system, then the C*-envelope of $E$ as a selfadjoint operator space coincides with the C*-envelope of $E$ as an operator system.

On the other hand the notion of just completely isometric complete order embeddings is not (in general) sufficient for an extremal theory, see \cite[Remark 2.27]{CS21}.
Below we apply \cite[Remark 2.27]{CS21} on a specific example that we will be carrying forwards.

\begin{example}\label{E:F_2} \cite[Remark 2.27]{CS21}
For $r\in (0,1] $, consider
\[
x_r := \begin{bmatrix} 1 & 0 \\ 0 & -r \end{bmatrix}.
\]
Set $E_r:=\spn\{x_r\}=\spn\{-x_r\}$. 
In $\ca(E_r) \cong\bC \oplus \bC$, we have
\[
\si_{\ca(E_r)}(x_r) = \{1, -r\}
\qand
\si_{\ca(E_r)}(-x_r) = \{-1, r\}.
\]
By construction, the map
\[
\phi \colon E_r \to E_r; \phi(x_r) = -x_r,
\]
is completely isometric and a complete order isomorphism since the positive cones are trivial. 

Suppose that the category of selfadjoint operator spaces with respect to completely isometric complete order embeddings admits a C*-envelope which we denote by $\C$. 
Then by the defining property of the C*-envelope, there are $*$-epimorphisms 
\[
\pi_1, \pi_2 \colon \ca(E_r)\to \C
\textup{ such that }
\pi_1(x_r) = \pi_2(-x_r).
\]
Hence we get
\[
\si_{\C}(\pi_1(x_r)) = \si_{\C}(\pi_2(x_r)) \subseteq \{1, -r\} \cap \{-1, r\}.
\]
By choosing $r \neq 1$, we arrive to the contradiction of an empty spectrum of an element. 

Note that, for any $r\in(0,1)$ the map 
\[
\psi \colon E_r \to E_1; \psi(x_r) = x_1,
\]
is a completely isometric complete order isomorphism, and thus $E_1$ does not admit a C*-envelope with respect to completely isometric complete order embeddings, as well.
\end{example}

We will also need Arveson's correspondence for selfadjoint operator spaces. 
In the following remark we recall that complete boundedness of the maps is preserved even in the non-unital setting.

\begin{remark}\label{R:Arv_cor}
Arveson's correspondence (see for example \cite[Proposition 13.2]{Pau02}) between positive functionals $s \colon M_n(E) \to \bC$ and completely positive maps $\phi \colon E \to M_n(\bC)$ is given by 
\[
s_\phi \colon M_n(E) \to \bC; s_\phi((x_{ij})) = \sum_{i,j} \sca{\phi(x_{ij})e_j, e_i}
\qand
\phi_s \colon E \to M_n(\bC); \phi_s(x) = (s(x\otimes E_{ij})),
\]
where $(e_i)_{i=1}^n$ is the canonical orthonormal basis of $\bC^n$, and $(E_{ij})_{i,j=1}^n$ is the family of matrix units in $M_n(\bC)$. 
By a well-known result of Smith \cite{Sm83} we have $\|\phi_s^{(n)}\|=\| \phi_s\|_{\rm cb}$ and $\|\phi_s^{(n)}\| \leq n \|\phi_s\|$, (see also  \cite[Exercise 3.10 and Proposition 8.11]{Pau02}). 
Hence
\[
\|\phi_s\|_{\rm cb}\leq n^3 \|s\|
\qand
\|s_\phi\|\leq n \|\phi\|_{\rm cb}.
\]
Note that Smith's result also implies that bounded linear maps taking values in $M_n(\bC)$ are automatically completely bounded.
\end{remark}

To fix notation, for a selfadjoint element $x$ in a C*-algebra $\C$ we write $x = x_{+} - x_-$ for the (unique) decomposition such that $x_{+}, x_-\in \C_{+}$ and $ x_{+} x_-=0$.
The following proposition appears also in \cite[Examples 3.3 and 4.3]{Rus23} and in \cite[proof of Proposition 6.3]{vD25}. 
We provide an alternative proof.

\begin{proposition} \label{P:gaugeC*} \cite[Examples 3.3 and 4.3]{Rus23} \cite[proof of Proposition 6.3]{vD25}
Let $\C \subseteq \B(H)$ be a C*-algebra and $x \in \C_{sa}$. 
Then
\[ 
\dist(x, -\C_{+})=\|x_{+}\|
\qand
\dist(x, \C_{+})=\|x_{-}\|.
\]
\end{proposition}

\begin{proof} 
It suffices to show the first equality; the second equality then follows by applying for $-x$. 
It is evident that 
\[
\dist(x, -\C_{+})
\leq
\|x + x_{-}\| = \|x_+\|.
\] 
On the other hand, by an application of \cite[Proposition 6.9]{Bre10}, for $p \in \C_{+}$ we get
\[
\|x_{+}\| \leq \|(x + p)_{+}\| \leq \|x + p\|.
\]
As $p$ was arbitrary we conclude $\|x_{+}\| \leq \dist(x, -\C_{+}) $.
\end{proof}

\section{Embeddings} \label{S:emb}

In this section we give a characterisation for a completely isometric complete order embedding between selfadjoint operator spaces to be an embedding. 
Our first goal is to provide  a formula that relates the contractive positive functionals of all levels to the distance to the cones of selfadjoint operator spaces. 
We will need the following proposition which can be deduced from \cite[Corollary 4.14]{vD25}; we give a direct proof that fits in our context. 

\begin{proposition} \label{P:dualcones}
Let $E\subseteq \B(H)$ be a selfadjoint operator space. If $x\in M_n(E)_{sa}$, then the following are equivalent:
\begin{enumerate}
\item $\mathrm{dist}(x, M_n(E)_+) \leq 1$.
\item $\vphi(x) \geq -1 $ for every  $\vphi \in {\rm CCP}(M_n(E), \bC)$.
\end{enumerate}
\end{proposition}

\begin{proof}
\noindent
[(i) $\Rightarrow$ (ii)]. 
Assume that $\mathrm{dist}(x, M_n(E)_+)\leq 1$ and let $\vphi\colon M_n(E) \to \bC$ be a contractive positive functional. 
We will prove that $\vphi(x) \geq -1$. 
Towards this end, fix $\eps>0$. 
Then there exists a $p \in M_n(E)_+$ such that $ \|x - p\|< 1+\eps$. We have 
\[
-\vphi(x-p) \leq |\vphi(x - p)| \leq \|x - p\|< 1+\eps,
\]
and thus we obtain
\[
\vphi(x) = \vphi(p) + \vphi(x - p) \geq \vphi(x-p)> -(1+\eps).
\]
Since $\eps$ was arbitrary we conclude that $\vphi(x)\geq -1$, as required.

\medskip

\noindent
[(ii) $\Rightarrow$ (i)]. 
Assume that $\vphi(x) \geq -1$  for every contractive positive functional $\vphi \colon M_n(E)\to \bC$. 
We will prove that $\mathrm{dist}(x, M_n(E)_+)\leq 1$. 
Set 
\[ 
\fC := \{ (x,t)\in M_n(E)_{sa}\oplus \bR : \dist(x, M_n(E)_{+}) \leq t\},
\]
and note that $\fC$ is  a closed convex cone in $ M_n(E)\oplus \bC$. 
We assume towards a contradiction that $(x,1) \notin \fC$. 
Without loss of generality by the Hahn--Banach Separation Theorem, we may pick a selfadjoint functional $f \colon M_n(E) \oplus \bC \to \bC$ such that 
\begin{align*}
f(x,1) < 0 \leq \inf_{(y,t)\in \fC} f(y,t). 
\end{align*}
Then $f$ decomposes as $f(y,t) = \psi(y) + t \mu$ for the selfadjoint map $\psi := f|_{M_n(E)}$ and $\mu := f(0,1) \in \bR$. 
Then $\psi$ is a positive linear functional on $M_n(E)$ and $\mu \in \bR_{+}$. 
Indeed, for $p \in M_n(E)_{+}$ we have $(p,0)\in \fC$, and so $0 \leq f(p,0) = \psi(p)$. 
Similarly, we have $(0, 1) \in \fC$ and hence $0 \leq f(0,1) = \mu$.

We now prove that $\|\psi\| \leq \mu$. 
Since $\psi$ is a selfadjoint functional it suffices to show that $|\psi(y)| \leq \mu \|y\|$ for every $y \in M_n(E)_{sa}$ (see for example \cite[page 92]{Mur90}). 
We have 
\[
\dist(y, M_n(E)_{+}) \leq \|y\| \qand \dist(-y, M_n(E)_{+}) \leq \|y\|,
\]
and hence $(-y, \|y\|)$ and $(y, \|y\|)$ are elements in $\fC$. 
It follows that 
\[
\psi(-y) + \mu \|y\| = f(-y, \|y\|) \geq 0
\qand
\psi(y) + \mu \|y\| = f(y, \|y\|) \geq 0,
\]
and thus
\[\
-\mu\|y\| \leq  \psi(y) \leq \mu \|y\|,
\]
which proves the claim. 

Finally, note that $\mu>0$, and hence for $\vphi:= \psi /\mu$ in ${\rm CCP}(M_n(E), \bC)$ we get
\[
\mu(\vphi(x)+1)=\psi(x) + \mu = f(x,1) <0.
\]
We derive $\vphi(x) < -1$, and thus a contradiction.
\end{proof}

We now obtain the desired distance formula. 
We  note that the equivalence [(i) $\Leftrightarrow$ (ii)] of \cite[Theorem 8.3]{vD25} can follow from this.

\begin{theorem}\label{T:stateformulaE}
Let $E \subseteq \B(H)$ be a selfadjoint operator space. 
If $x \in M_n(E)_{sa}$, then we have
\[
\dist(x,M_n(E)_{+}) = \sup_{\vphi \in {\rm CCP}(M_n(E), \bC)} -\vphi(x).
\]
\end{theorem} 

\begin{proof}
It follows by Proposition \ref{P:dualcones} and a normalisation argument.
\end{proof}

We will connect embeddings to the notion of maximal gauges introduced by Russell \cite{Rus23}. 
Let us recall here the terminology.
Let $\vthe \colon E\to F$ be a completely isometric complete order embedding between selfadjoint operator spaces. 
Then the map $\vthe$ is called a \textit{gauge maximal isometry} if for every $n\in\bN$ we have 
\[
\dist(x,M_n(E)_+) =  \dist(\vthe^{(n)}(x),M_n(F)_+)
\foral
x\in M_n(E)_{sa}.
\]
Following \cite[Definition 4.2]{Rus23}, whenever $E$ and $F$ are gauge maximal isometric via the inclusion map $E\subseteq F$ we will just say that we have a \emph{gauge maximal inclusion}.

\begin{theorem} \label{T:equivemb}
Let $\vthe \colon E\to F$ be a completely isometric complete order embedding between selfadjoint operator spaces. 
Then the following are equivalent:
\begin{enumerate}
\item The map $\vthe$ is an embedding.
\item Every map $\vphi \in {\rm CCP}(M_n(E), \bC)$ extends to a map $\wt{\vphi} \in {\rm CCP}(M_n(F), \bC)$ with $\wt{\vphi} \circ \vthe =\vphi $, for every $n\in \bN$.
\item The map $\vthe$ is a gauge maximal isometry.
\end{enumerate}
\end{theorem}

\begin{proof}
Without loss of generality we may assume that $E \subseteq F$.

\smallskip

\noindent
[(i) $\Leftrightarrow$ (ii)]. 
This equivalence is shown in \cite[Lemma 2.2]{DKP25}.

\smallskip

\noindent
[(ii) $\Rightarrow$ (iii)]. 
The assumption gives that, for every $x\in M_n(E)_{sa}$ we have 
\[
\sup_{\vphi \in {\rm CCP}(M_n(F), \bC)} -\vphi(x)
=
\sup_{\vphi \in {\rm CCP}(M_n(E), \bC)} -\vphi(x),
\]
and hence Theorem \ref{T:stateformulaE} yields
\[
\dist(x,M_n(F)_+)
=
\dist(x,M_n(E)_+).
\]

\smallskip

\noindent
[(iii) $\Rightarrow$ (i)]. 
Consider $(x,a)\in M_n(E^\#)_{sa}$ such that $(x, a)\in M_{n}(F^{\#})_{+}$.
Then $a \geq 0$ and 
\[
\sup_{\vphi \in {\rm CCP}(M_n(F), \bC)} -\vphi(a_{\eps}^{-1/2} x a_{\eps}^{-1/2}) \leq 1 \foral \eps>0.
\]
Since $E\subseteq F$ is a gauge maximal inclusion, for every $\eps>0$ we have
\begin{align*}
\sup_{\vphi \in {\rm CCP}(M_n(E), \bC)} -\vphi(a_{\eps}^{-1/2} x a_{\eps}^{-1/2})
&=
\dist(a_{\eps}^{-1/2} x a_{\eps}^{-1/2},M_n(E)_+)\\
&=
\dist(a_{\eps}^{-1/2} x a_{\eps}^{-1/2}, M_n(F)_+)\\
&=
\sup_{\vphi \in {\rm CCP}(M_n(F), \bC)} -\vphi(a_{\eps}^{-1/2} x a_{\eps}^{-1/2})
\leq 1,
\end{align*}
where we used Theorem \ref{T:stateformulaE} in the first and third equalities.
Hence $(x, a)\in M_{n}(E^{\#})_{+}$, as required.
\end{proof}

We thus obtain the following corollary.

\begin{corollary} \label{C:spec}
Let $E$ be a selfadjoint operator space in a C*-algebra $\C$.
Then the inclusion $E \subseteq \C$ is an embedding if and only if for every $x \in M_n(E)_{sa}$ we have
\[
\dist(x, M_n(E)_+) = \max_{\la \in \si_{M_n(\C)}(x)} -\la.
\]
\end{corollary}

\begin{proof}
This follows from Proposition \ref{P:gaugeC*} and the equivalence [(i) $\Leftrightarrow$ (iii)] of Theorem \ref{T:equivemb}.
\end{proof}

As a further corollary we obtain \cite[Theorem 4.4 and Theorem 4.5]{Rus23}.

\begin{corollary}\label{C:incl} \cite[Theorem 4.4 and Theorem 4.5]{Rus23}
The following hold:
\begin{enumerate}
\item If $\S\subseteq \T$ are operator systems with $1_\S=1_\T$, then $\S\subseteq \T$ is a gauge maximal inclusion.
\item If $\C$ is a C*-subalgebra of a C*-algebra $\D$ then $\C\subseteq \D$ is a gauge maximal inclusion.
\end{enumerate}
\end{corollary}

\begin{proof}
Both cases follow from Theorem \ref{T:equivemb} since we can extend contractive positive functionals from the subspace to the superspace at any matrix level.
\end{proof}

In Corollary \ref{C:emb} we consider the non-unital inclusion of operator systems.
Using Theorem \ref{T:equivemb} we also obtain an alternative proof of \cite[Corollary 4.17]{Wer02}. 

\begin{corollary}\label{C:C*-unit}\cite[Corollary 4.17]{Wer02}
Let $\C$ be a C*-algebra. 
Then the operator system unitisation of $\C$ is unital completely order isomorphic to the C*-algebraic unitisation of $\C$.
\end{corollary}

\begin{proof}
For convenience, let us denote the C*-algebraic unitisation of a C*-algebra $\C$ by $\wt{\C}$.
By Corollary \ref{C:incl} the inclusion $\C \subseteq \wt{\C}$ is an embedding, and hence by Theorem \ref{T:equivemb} the canonical map $\io \colon \C^{\#} \to (\wt{\C})^{\#}$ is a unital complete order embedding. 
As $\wt{\C}$ is an operator system, by \cite[Lemma 4.9(b)]{Wer02} we have $ (\wt{\C})^{\#} \cong \wt{\C}\oplus \bC$ via $(c, a) \mapsto(c + a e, a)$, where the direct sum is in the $\ell^{\infty}$ sense and $e$ is the adjoined unit of $\wt{\C}$. 
Now let $\Phi \colon \C^\# \to \wt{\C}$  be the composition of $\io$ by the projection to the first coordinate of $ \wt{\C} \oplus \bC$, that is
\[
\Phi(c, a) := c + a e \text{ for } (c, a) \in \C^\#.
\] 
Then $\Phi$ is a bijection, and a unital completely positive map as a composition of unital completely positive maps. 
On the other hand, consider $c + a \otimes e \in M_n(\wt{\C})_{+}$ for $a \in M_n(\bC)$. 
Then we have $a \in M_n(\bC)_{+}$, and thus 
\[
\io^{(n)}(c, a)=(c + a \otimes e, a) \in M_n((\wt{\C})^{\#})_{+}.
\]
Since $\io$ is a complete order embedding, we conclude that $(c, a) \in M_n(\C^{\#})_{+}$. 
\end{proof}

\begin{example} \label{E:exsp}
Theorem \ref{T:equivemb} provides a criterion for detecting inclusions that are not embeddings.
Let $E$ be a selfadjoint operator space inside a C*-algebra $\C$.
If there is an $x \in M_n(E)_{sa}$ and a $\vphi \in {\rm CCP}(M_n(E), \bC)$ with $\|\vphi\| = 1$ such that $\vphi(x) \notin \ol{{\rm conv}}\{\si_{\C}(x)\}$, then the inclusion $E \subseteq \C$ is not an embedding.
Indeed, if the inclusion $E \subseteq \C$ were an embedding, then $\vphi$ would have an extension $\wt{\vphi} \in {\rm CCP}(M_n(\C), \bC)$.
As $x$ is selfadjoint it follows that $\vphi(x) = \wt{\vphi}(x) \in \ol{{\rm conv}}\{\si_{\C}(x)\}$ (see for example \cite[Lemma 2.10]{Pau02}), which is a contradiction.
\end{example}

As another application of Theorem \ref{T:equivemb} we recover \cite[Corollary 4.2]{Rus23} and \cite[Lemma 4.2 (a)]{Ng11}.
We note that the equivalence [(i) $\Leftrightarrow$ (ii)] below is shown to hold in the broader context of MOS-embeddings in \cite[Lemma 4.2 (a)]{Ng11}.

\begin{theorem} \label{T:arvemb} \cite[Corollary 4.2]{Rus23}
Let $\vthe \colon E\to F$ be a completely isometric complete order embedding between selfadjoint operator spaces.
The following are equivalent:
\begin{enumerate}
\item The map $\vthe$ is an embedding.
\item Every map $\phi \in {\rm CCP}(E, M_k(\bC))$ extends to a map $\wt{\phi}\in {\rm CCP}( F, M_k(\bC))$ with $\wt{\phi}\circ \vthe = \phi$, for every $k\in \bN$.
\item Every map $\phi\in {\rm CCP}(E ,\B(K))$ extends to a map $\wt{\phi}\in {\rm CCP}(F , \B(K))$ with $\wt{\phi}\circ \vthe = \phi$, for every Hilbert space $K$.
\end{enumerate}
\end{theorem}

\begin{proof}
Without loss of generality we may assume that $E \subseteq F$.

\smallskip

\noindent
[(i) $\Rightarrow$ (ii)]. 
By \cite[Lemma 4.6]{Wer02}, every completely contractive completely positive map $ \phi \colon E \to M_k(\bC)$ extends to a unital completely positive map $ \phi^{\#}\colon E^{\#} \to M_k(\bC)$. 
Since $E^\# \subseteq F^\#$ is a unital inclusion of operator systems, by Arveson's Extension Theorem, the map $\phi^{\#}$ extends to a unital completely positive map $\wt{\phi^{\#}}\colon F^{\#} \to M_k(\bC)$ which in turn restricts to a completely contractive completely positive map $\wt{\phi}\colon F \to M_k(\bC)$ that extends $\phi$.

\medskip

\noindent
[(ii) $\Rightarrow$ (iii)].
Consider a completely contractive completely positive map $\phi \colon E \to \B(K)$. 
We write $(\phi_{\F})_{\F}$ for the net of compressions of $\phi$ into the finite dimensional subspaces of $K$ indexed by inclusion, i.e., $\phi_{\F}(x)=P_{\F}\phi(x)P_{\F} $ for the projection $P_{\F}$ onto $\F \subseteq K$.
Note that $\phi_{\F}$ acts on $\B(P_{\F} K)\cong M_k(\bC)$ where $k= \rank P_{\F}$, and it is a completely contractive completely positive map. 
By assumption, it extends to a completely contractive completely positive map 
\[
 \wt{\phi}_{\F} \colon F \to M_{k}(\bC) \subseteq \B(K).
\] 
Since the set ${\rm CCP}(F, \B(K))$
is compact in the BW-topology (see \cite[Theorem 7.4]{Pau02}), the net $(\wt{\phi}_{\F})_{\F}\subseteq {\rm CCP}(F, \B(K))$ has a cluster point $\wt{\phi}$, and take a subnet converging to it.
By definition $\wt{\phi}$ is a completely contractive completely positive extension of $\phi$ on $F$.

\medskip

\noindent
[(iii) $\Rightarrow$ (i)].
The proof of this implication is essentially the “only if” part of \cite[Theorem 4.3]{Rus23}. 
We include a proof that fits in our setting. 
By Theorem \ref{T:equivemb}, and since $M_n(E)_+\subseteq M_n(F)_+$, we only need to prove that 
\[
\dist(x, M_n(E)_{+})
\leq
\dist(x, M_n(F)_{+}).
\]
Let $j \colon E^{\#} \to \B(H)$ be a unital complete order embedding. 
By Remark \ref{R:covers} the map
\[
\phi:=j|_E\colon E\to\B(H)
\]
is an embedding and hence it is a gauge maximal isometry by Theorem \ref{T:equivemb}. 
By assumption, the map $\phi$ extends to a completely contractive completely positive map $\wt{\phi}\colon F \to \B(H)$. 
For $p \in M_n(F)_{+}$ we have $\wt{\phi}^{(n)}(p) \in M_n(\B(H))_{+}$, and therefore
\begin{align*}
\dist(x, M_n(E)_{+})
=
\dist(\phi^{(n)}(x), M_n(\B(H))_{+})
\leq 
\|\wt{\phi}^{(n)}(x) - \wt{\phi}^{(n)}(p) \|
\leq 
\|x-p\|.
\end{align*}
Taking the infimum over all $p \in M_n(F)_{+}$ completes the proof.
\end{proof}

\begin{example}\label{E:torus}
Let $\bT$ denote the unit circle in the complex plane and let $ {\rm C}(\bT)$ denote the C*-algebra of continuous functions on $\bT$. Define the selfadjoint operator space 
\[
E:= \{ f\in {\rm C}(\bT): f(z)= \al z + \be\bar z , \; \al, \be \in \bC \},
\] 
and note that $\ca(E)={\rm C}(\bT)$. 
Then the inclusion $E\subseteq {\rm C}(\bT)$ is an embedding. 

By Theorem \ref{T:arvemb} it suffices to show that any completely contractive completely positive map $\phi\colon E \to \B(K)$ admits a completely contractive completely positive extension $\wt{\phi} \colon {\rm C}(\bT) \to \B(K)$.
We begin by showing that the positive cone $E_{+}$ is trivial. 
Indeed, any $f \in E_+$ satisfies $ \al =\bar \be$ and thus $ f(z)= 2 {\rm Re}( \al z)$ for every $z\in \bT$. By writing $\al = r e^{i \theta} $ with  $ r \geq 0 $ and $\theta \in [0,2\pi)$, we have $0 \leq f(-e^{-i\theta})= 2{\rm{Re}}(\al (- e^{-i \theta})) = -2 r$. 
This yields $\al=0$ and thus $f=0$.

Therefore every completely contractive complete order embedding $\phi \colon E \to \B(K)$ is uniquely determined by the contraction $T := \phi (z)$, for the coordinate function $z$ in ${\rm C}(\bT)$. 
By Sz.-Nagy's dilation theorem, let $U \in \B(L)$ be a unitary dilation of $T$ for $L \supseteq K$. 
By the universal property of the full group C*-algebra $\ca(\bZ)\cong {\rm C}(\bT)$ we obtain a unital $*$-homomorphism 
\[
\Phi \colon {\rm C}(\bT) \to \B(L); \Phi(z) := U,
\]
and the unital (completely contractive) completely positive map $P_K \Phi|_K$ is by definition an extension of $\phi$ on ${\rm C}(\bT)$.
Thus by Theorem \ref{T:arvemb} the inclusion $E \subseteq {\rm C}(\bT)$ is an embedding.

More generally, let $ \bF_{n}$ denote the free group on $n$ generators, and let $E_{n}$ be the selfadjoint operator space spanned by the unitary generators of the full group C*-algebra $\ca(\bF_{n})$, i.e., $E_{n}:= \spn \{u_{i}, u^{*}_{j}: i,j =1,\dots , n \}$. 
Then the inclusion $ E_{n} \subseteq \ca(\bF_{n})$ is an embedding.

By Theorem \ref{T:arvemb} it suffices to show that any completely contractive completely positive map $\phi\colon E_{n} \to \B(K)$ admits a completely contractive completely positive extension on $\ca(\bF_{n})$. 
For $i =1,\dots , n$ consider the contractions $v_{i}:=\phi(u_{i})\in \B(K)$. 
Consider also the unitary dilations 
\[
\wh{v_i}:= \begin{bmatrix}
v_i & (1 - v_iv_i^*)^{1/2}\\
( 1- v_i^*v_i)^{1/2} & - v_i^*
\end{bmatrix} \in \B(K^{(2)}).
\]
By the universal property of $\ca(\bF_{n})$, we obtain a unital $*$-homomorphism 
\[
\Phi\colon \ca(\bF_n ) \to \B(K^{(2)}); \Phi(u_i) = \wh{v_i}.
\]
Then the compression $P_K \Phi|_K$ is a unital (completely contractive) completely positive extension of $\phi$ on $\ca(\bF_n)$, as required.
\end{example}

We end this section with the connection between Werner's unitisation of matrix ordered operator spaces and Russell's unitisation of matrix gauge $*$-vector spaces. 
Following \cite[Definition 3.1]{Rus23}, a \emph{proper matrix gauge} on a $*$-vector space $V$ is a sequence of functions 
\[
\nu_n \colon M_n(V)_{sa} \to [0,\infty), n \in \bN,
\]
that satisfy the following conditions:
\begin{enumerate}
\item If $x,y\in M_n(V)_{sa}$, then $\nu_n(x+y)\leq \nu_n(x)+\nu_n(y)$.
\item If $x\in M_n(V)_{sa}$ and $a\in M_{n,k}$, then $\nu_k(a^*xa)\leq \|a\|^2 \nu_n(x)$.
\item If $x\in M_n(V)_{sa}$ and $y\in M_m(V)_{sa}$, then $\nu_{n+m}(x\oplus y)=\max\{\nu_n(x), \nu_m(y)\}$.
\item If $x\in M_n(V)_{sa}$ and $\nu_n(x)=\nu_n(-x)=0$, then $x=0$.
\end{enumerate}
The pair $(V,\nu)$ is called a \emph{matrix gauge $*$-vector space}.
In \cite[Corollary 4.1]{Rus23} it is shown that any matrix gauge $*$-vector space $(V,\nu)$ is a selfadjoint operator space in our sense, with the positive cones defined by
\[
C_n:=\{x\in M_n(V)_{sa} : \nu_n(-x)=0\} \foral n \in \bN,
\]
and the norms defined by
\[
\|x\|_n:=\max\left\{\nu_{2n}\left(\begin{bmatrix}
0 & x \\ x^* & 0 
\end{bmatrix}\right), \nu_{2n}\left(\begin{bmatrix}
0 & -x \\ -x^* & 0 
\end{bmatrix}\right)\right\}
\foral x\in M_n(V) \text{ and } n\in\bN. 
\]
Moreover, in \cite[Theorem 4.1]{Rus23} it is proven that if $E\subseteq \B(H)$ is a selfadjoint operator space, then the \emph{maximal gauge} defined by 
\[
\nu_n^{\max}(x):= \dist(x,-M_n(E)_+) \text{ for every } x\in M_n(E)_{sa} \text{ and } n\in \bN,
\]
defines a proper matrix gauge on $E$, and the induced matrix cones and norms from this gauge coincide with the matrix cones and norms inherited from $\B(H)$.
A map $\phi \colon (V,\nu) \to (W,\om)$ between matrix gauge $*$-vector spaces is called \emph{completely gauge contractive} (resp. \emph{isometric}) if $\om_n(\phi^{(n)}(x)) \leq \nu_n(x)$ (resp. $\om_n(\phi^{(n)}(x)) =\nu_n(x)$) for every $x\in M_n(V)_{sa}$ and $n\in \bN$.

For a matrix gauge $*$-vector space $(V,\nu)$ Russell \cite[Definition 3.4]{Rus23} defines a proper matrix gauge $\wh{\nu}$ on the $*$-vector space $\wh{V}:=V\oplus \bC$ by setting
\[
\nu_n(x,a):=\inf\{ t>0 : 0 \leq tI_n-a \in  {\rm Inv}(M_n(\bC)) \text{ and } \nu_n((tI_n-a)^{-1/2} x (tI_n-a)^{-1/2}))\leq 1\},
\]
for every $(x, a)\in M_n(\wh{V})_{sa}$ and every $n\in \bN$.
In \cite[Theorem 3.1]{Rus23} it is shown that the unitisation $(\wh{V},\wh{\nu})$  is an operator system and that the inclusion map 
\[
V \hookrightarrow \wh{V}; x \mapsto (x,0),
\]
is completely gauge isometric. 
Moreover, in \cite[Corollary 3.2]{Rus23} it is proven that if $\S$ is an operator system with unit $e$ and $\phi\colon V \to \S$ is completely gauge contractive with respect to the maximal gauge of $\S$, then the map 
\[
\wh{\phi} \colon \wh{V} \to \S ; (x,a) \mapsto \phi(x)+a e,
\]
is unital completely positive.
The following corollary shows the canonical connection between Werner's unitisation and Russell's unitisation of $E$ when endowed with the maximal gauge.

\begin{corollary} \label{C:Wun=Run}
Let $E$ be a selfadjoint operator space and let $(\wh{E},\wh{\nu}^{\max})$ be Russell's unitisation of the maximal gauge defined on $E$. 
Then $E^\#$ and $\wh{E}$ are isomorphic operator systems.
\end{corollary}

\begin{proof}
The inclusion map $E \hookrightarrow \wh{E}$ is completely gauge isometric by construction of $(\wh{E}, \wh{\nu}^{\max})$, and hence by \cite[Proposition 3.3]{Rus23} it is completely contractive completely positive. We thus obtain a unital completely positive map 
\[
\Phi \colon E^\# \to \wh{E}; (x,a) \mapsto (x,a).
\]

Conversely, the inclusion map $E \hookrightarrow E^\#$ is completely gauge isometric with respect to the maximal gauges of $E$ and $E^\#$. Indeed, by Theorem \ref{T:stateformulaE} we have 
\begin{align*}
\dist(x,M_n(E)_+)
&=
\sup_{\vphi \in {\rm CCP}(M_n(E), \bC)} -\vphi(x)
&=
\sup_{\vphi \in {\rm CCP}(M_n(E^\#), \bC)} -\vphi(x)
=
\dist(x, M_n(E^\#)_+),
\end{align*}
for all $x \in M_n(E)_{sa}$, $n\in \bN$.
Hence, \cite[Corollary 3.2]{Rus23} implies that the map 
\[
\Psi \colon \wh{E} \to E^\#; (x,a) \mapsto (x,a),
\]
is unital completely positive. 
Since $\Phi$ and $\Psi$ are mutual inverses, the proof is complete.
\end{proof}

\section{Positive generation} \label{S:pos gen}

In this section we investigate the embedding problem for selfadjoint operator spaces that are generated by their cones (at every matrix level).
Recall that a selfadjoint operator space $E$ is called \emph{approximately positively generated} if 
\[
E_{sa} = \ol{E_+-E_+}.
\]
In \cite[Example 8.1]{HKM23}, the authors provide an example of a selfadjoint operator space $E$ that is approximately positively generated, but for which $E_+ - E_+$ is not closed; hence $E$ is not positively generated.

We will relate this property to rigidity of the zero representation, a notion introduced by Salomon \cite[Definition 3.1]{Sal17} in the broader context of generating sets.
Let $\G \subseteq \ca(\G)$ be a generating set. 
The set $\G$ is said to be \emph{rigid at zero} if for any sequence $(\vphi_n)_n \subseteq {\rm CCP}(\ca(\G), \bC)$ we have:
\[
\lim_n \vphi_n(g) = 0 \foral g \in \G
\; \Longrightarrow \;
\lim_n \vphi_n(c) = 0 \foral c \in \ca(\G).
\]
In \cite[Theorem 3.3]{Sal17} it is shown that $\G$ is rigid at zero if and only if it is \emph{separating} for $\ca(\G)$, i.e., if for any $\vphi \in {\rm CCP}(\ca(\G), \bC)$ we have
\[
\vphi(g) = 0 \foral g \in \G
\; \Longrightarrow \;
\vphi(c) = 0 \foral c \in \ca(\G).
\]

Let now $E \subseteq \ca(E)$ be a selfadjoint operator space.
A map $\phi \colon E \to \B(K)$ is called \emph{maximal} if $\phi$ is a direct summand of any completely contractive completely positive dilation.
A map $\phi \colon E \to \B(K)$ is said to have \emph{the unique extension property} if it has a unique completely contractive completely positive extension $\wt{\phi} \colon \ca(E) \to \B(K)$ that is a $*$-representation.

As shown in \cite{Sal17}, the condition of rigidity at zero links with the study of hyperrigidity of non-unital generating sets, see also \cite{DKP25}. 
Moreover, in \cite[Theorem 3.3]{Sal17} it is proved that it is equivalent to the zero representation having the unique extension property. 
Combining the results of \cite{HKM23, Kim24, Sal17} we obtain the following.

\begin{proposition} \label{P:sepequiv} \cite[Proposition 8.2]{HKM23}, \cite[Theorem A]{Kim24}, \cite[Theorem 3.3]{Sal17}
Let $E \subseteq \ca(E)$ be a selfadjoint operator space and let $\de_0\colon E \to \bC$ denote the zero functional. 
Consider the following statements:
\begin{enumerate}
\item $E$ is approximately positively generated.
\item $E_+$ separates ${\rm CCP}(E, \bC)$.
\item $\de_0$ is maximal.
\item $\de_0$ satisfies the unique extension property.
\item $E$ is rigid at zero.
\item $E$ is separating for $\ca(E)$.
\end{enumerate}
Then the following implications hold:
\[
[\textup{(i) $\Leftrightarrow$ (ii) $\Leftrightarrow$ (iii) $\Rightarrow$ (iv) $\Leftrightarrow$ (v) $\Leftrightarrow$ (vi)}].
\]

If in addition the inclusion $E \subseteq \ca(E)$ is an embedding, then all items above are equivalent.
\end{proposition}

\begin{proof}
Equivalence [(i) $\Leftrightarrow$ (ii)] is \cite[Proposition 8.2]{HKM23}. 
Equivalence [(i) $\Leftrightarrow$ (iii)] is \cite[Theorem A]{Kim24}. 
Equivalences [(iv) $\Leftrightarrow$ (v) $\Leftrightarrow$ (vi)] is \cite[Theorem 3.3]{Sal17}. 
The implication [(iii) $\Rightarrow$ (iv)] and the conditional implication [(iv) $\Rightarrow$ (iii)] follow by standard arguments.

In more detail, suppose that item (iii) holds.
First note that $\de_0$ has an extension to the zero representation on $\ca(E)$.
Let $\wt{\de}_0 \colon \ca(E) \to \bC$ be another contractive positive extension of $\de_0$.
Then passing to the unitisation and applying Stinespring's Theorem (see for example \cite[Proposition 2.7]{DKP25}) we can find a $*$-representation $\pi \colon \ca(E) \to \B(K)$ that dilates $\wt{\de}_0$.
By maximality we have that $\de_0$ is a direct summand of $\pi|_E$.
Since $\pi$ is a $*$-representation we see that its (1,1) corner $\wt{\de}_0$ is zero on $\ca(E)$, showing that $\de_0$ has the unique extension property.

Finally assume that $E \subseteq \ca(E)$ is an embedding.
We will show that in this case item (iv) implies item (iii).
Let $\psi \colon E \to \B(K)$ be a completely contractive completely positive dilation of $\de_0$.
As the inclusion $E \subseteq \ca(E)$ is assumed to be an embedding, by Theorem \ref{T:arvemb} there is a completely contractive completely positive extension $\wt{\psi} \colon \ca(E) \to \B(K)$ of $\psi$.
Since $P_\bC \wt{\psi} |_{\bC}$ is a contractive positive extension of $\de_0$, by the unique extension property we get $P_\bC \wt{\psi} |_{\bC} = 0$.
Since $\ca(E)$ is generated by its positive elements we have then that the off-diagonal corners of $\wt{\psi}(\ca(E))$ are zero, and hence so is for $\psi(E)$.
Thus $\de_0$ is a direct summand of $\psi$ as required.
\end{proof}

We next show that separation at the first level is separation at every level.

\begin{proposition}\label{P:complsep}
Let $E\subseteq \ca(E)$ be selfadjoint operator space. The following are equivalent:
\begin{enumerate}
\item $E$ is separating for $\ca(E)$.
\item $M_n(E)$ is separating for $M_n(\ca(E))$ for all $n\in \bN$.
\end{enumerate}
\end{proposition}

\begin{proof}
The implication [(ii) $\Rightarrow$ (i)] is trivial. 
For [(i) $\Rightarrow$ (ii)], let $\vphi\colon M_n(\ca(E)) \to \bC$ be a contractive positive functional such that $\vphi|_{M_n(E)} = 0$.
As in \cite[Lemma 2.7]{Wer02} there exist $\vphi_{ij} \colon E \to \bC$ such that $\vphi_{ii}$ are in ${\rm CCP}(\ca(E), \bC)$ and 
\[
\vphi(c)=\sum_{i,j} \vphi_{ij}(c_{ij}) \text{ for } c=(c_{ij})\in M_n(\ca(E)).
\]
Since $\vphi|_{M_n(E)}=0$ we have $\vphi_{ii}|_E=0$ and hence $\vphi_{ii}=0$ for every $i=1,\dots,n$. 
Let $(e_\la)_\la$ be an approximate unit for $\ca(E)$, and thus $(e_\la \otimes I_n)_\la$ is an approximate unit for $M_n(\ca(E))$. 
Since $\vphi \in {\rm CCP}(M_n(\ca(E)), \bC)$ we obtain 
\[
\|\vphi\| = \lim_\la \vphi(e_\la \otimes I_n) = \lim_\la \sum_i \vphi_{ii}(e_\la) = 0,
\]
and thus $M_n(E)$ is separating for $\ca(M_n(E))$, as required.
\end{proof}

In \cite[Proposition 8.4]{HKM23} it is shown that if $E$ is \emph{positively generated}, i.e., $E_{sa}=E_+-E_+$, then $M_n(E)$ is positively generated for every $n\in\bN$. 
Using Propositions \ref{P:sepequiv} and \ref{P:complsep} we can show that the same result holds for approximate positive generation.

\begin{proposition}\label{P:complappos}
Let $E$ be a selfadjoint operator space. 
The following are equivalent:
\begin{enumerate}
\item $E$ is approximately positively generated.
\item $M_n(E)$ is approximately positively generated for all $n\in \bN$.
\end{enumerate}
\end{proposition}

\begin{proof}
The implication [(ii) $\Rightarrow$ (i)] is trivial. 
For [(i) $\Rightarrow$ (ii)], we may consider an inclusion $E\subseteq \ca(E)$ which is an embedding, use for example Remark \ref{R:covers}. 
Then for every $n\in \bN$ we have that the inclusion $M_n(E) \subseteq M_n(\ca(E))$ is an embedding as well. 
By Proposition \ref{P:complsep} we have that $M_n(E)$ is separating for $M_n(\ca(E))$, and hence Proposition \ref{P:sepequiv} yields that $M_n(E)$ is approximately positively generated.
\end{proof}

We will prove that approximate positive generation allows extensions of completely positive maps when the generated C*-algebra is unital. 
We will first need the following lemma.

\begin{lemma}\label{L:extendable_closed} Let $E \subseteq \ca(E)$ be a selfadjoint operator space. 
If $E$ is separating for $\ca(E)$ and $\ca(E)$ is unital, then the set 
\begin{equation} \label{eq:extendables}
\fK:=\big\{ \phi\in {\rm CB}(E,\B(K)) : \exists \,  \wt{\phi} \colon \ca(E) \to \B(K) \text{ completely positive such that } \wt{\phi}|_{E}=\phi \big\}
\end{equation}
of completely bounded completely positive extendable maps from $E$ to $\ca(E)$ is closed under the point-weak* topology.
\end{lemma}

\begin{proof}
Since $ \fK$ is a convex cone, by the Krein--\v{S}mulian Theorem it suffices to show that  
\[
\fK \cap \{\phi \in {\rm CB}(E,\B(K)): \| \phi \|_{\rm cb}\leq 1 \}
\]
is point-weak* closed. 
Let $(\phi_{\la})_{\la\in \La} \subseteq \fK $ be a net in $\fK$ such that $\|\phi_{\la}\|_{\rm cb}\leq 1$ for all $\la$ and $\textup{pw*-}\lim \phi_{\la} = \phi$ for some completely bounded and completely positive map $\phi\colon E \to \B(K)$. 
For each $\la$ let $\wt{\phi}_\la\colon \ca(E) \to \B(K)$ be a completely positive extension of $\phi_\la$. 
We will prove that $(\wt{\phi}_{\la})_{\la}$ has a uniformly bounded subnet. 
By compactness of bounded sets in the point-weak* topology, this will yield a cluster point $\wt{\phi} \colon \ca(E) \to \B(K)$ which will be a completely positive extension of $\phi$, and the proof will be complete.

We assume towards a contradiction that the net $(\|\wt{\phi}_{\la}\|_{\rm cb})_{\la}$ does not have a uniformly bounded subnet.
That is, for every $M>0$ and every $\la_0\in \La$ there exists $ \la \geq \la_0$ such that $\|\wt{\phi}_{\la }\|_{\rm cb} >M $. 
By taking a subnet if necessary we may assume that  $ \| \wt{\phi}_{\la}\|_{\rm cb} \to \infty$.
Now for each $\la$, pick a  state $\om_{\la}\colon \B(K)\to \bC$ with $\om_{\la}(\wt{\phi}_{\la}(1))
=
\|\wt{\phi}_{\la}(1)\|
=
\|\wt{\phi}_{\la}\|_{\rm cb}$,
and set  
\[
f_{\la }:= \|\wt{\phi}_{\la}\|_{\rm cb}^{-1} \cdot (\om_{\la} \circ \wt{\phi}_{\la}). 
\]
Each  $ f_{\la}$ is a state on the unital $\ca(E)$, and hence by weak* compactness, there exists a state $f$ of $\ca(E)$ that is a cluster point for $(f_\la)_\la$. 
For $x \in E$ we have
\begin{align*}
| f_{\la}(x)| = \|\wt{\phi}_{\la}\|_{\rm cb}^{-1} \cdot |\om_{\la}(\wt{\phi}_{\la}(x))| = \|\wt{\phi}_{\la}\|_{\rm cb}^{-1} \cdot |\om_{\la}(\phi_{\la}(x))| \leq \|\wt{\phi}_{\la}\|_{\rm cb}^{-1} \cdot \|x\| \xrightarrow{\la} 0
\end{align*}
since $ \|\om_{\la}\| =1$ for all $\la$ and $\|\phi_{\la}\|_{\rm cb}\leq 1$. 
We thus obtain that $f|_{E}=0$, and since $E$ is separating for $\ca(E)$, we derive to the contradiction that $f = 0$.
\end{proof}

We now prove the extension theorem in this class.

\begin{theorem}\label{T:exten}
Let $E \subseteq \ca(E)$ be a selfadjoint operator space. 
If $E$ is separating for $\ca(E)$ and $\ca(E)$ is unital, then every completely bounded completely positive map $\phi\colon E \to \B(K)$ extends to a (completely bounded) completely positive map $\wt{\phi} \colon \ca(E) \to \B(K)$.
\end{theorem}

\begin{proof}
First let $\vphi \colon E \to \bC$ be a bounded positive functional.
By Lemma \ref{L:extendable_closed}, the set
\[
\fK_{\bC} :=\{\vphi \in E^d : \exists \, \wt{\vphi} \colon \ca(E) \to \bC \text{ positive such that } \wt{\vphi}|_{E} = \vphi\},
\]
is a weak*-closed convex subset of $E^d$. 
We claim that $\vphi \in \fK_{\bC}$. 
Indeed, if not, without loss of generality by the Hahn--Banach Separation Theorem there would have to be some selfadjoint element $x\in E$ for which
\[
\vphi(x)< 0 \leq \inf_{\vphi' \in \fK_{\bC}} \vphi'(x).
\]
Since the restriction to $E$ of every positive functional of $\ca(E)$ lies in $\fK_{\bC}$, we obtain that $x\geq 0$, which contradicts to the fact that $\vphi$ is positive.

Next let $\phi \colon E \to M_n(\bC)$ be a completely bounded completely positive map.
By Arveson's correspondence (see Remark \ref{R:Arv_cor}) we obtain a bounded positive functional $s_\phi \colon M_n(E) \to \bC$. 
By Proposition \ref{P:complappos} we have that $M_n(E)$ is also approximately positively generated, and thus we can extend  $s_\phi$ to a positive functional on $M_n(\ca(E))$.
Using Arveson's correspondence once more, we obtain the desired completely bounded completely positive extension of $\phi \colon E \to M_n(\bC)$.

Finally, let $\phi \colon E \to \B(K)$ be a completely bounded completely positive map.
Then $\phi$ is the point-weak* limit  of the net of compressions $(\phi_{\F})_{\F}$ of $\phi$ onto the finite dimensional subspaces of $K$ indexed by inclusion, i.e., $\phi_{\F}= P_{\F} \phi P_{\F} $ where $P_{\F}$ is the projection onto $\F \subseteq K$. 
Each of the maps $\phi_{\F}$ belongs to the set $\fK$ of ``extendable" completely positive maps in (\ref{eq:extendables}) since $\B(P_{\F} K)\cong M_{k}(\bC)$ where $k= \rank P_{\F}$. 
By Lemma \ref{L:extendable_closed}, the set $\fK$ is closed under the point-weak* topology, thus showing that $\phi \in \fK$, as required.
\end{proof}

By Proposition \ref{P:sepequiv} and Theorem \ref{T:exten} we obtain the following corollary.

\begin{corollary}\label{C:exten}
If $E \subseteq \ca(E)$ is an approximately positively generated selfadjoint operator space and $\ca(E)$ is unital, then every completely bounded completely positive map $\phi\colon E \to \B(K)$ extends to a (completely bounded) completely positive map $\wt{\phi} \colon \ca(E) \to \B(K)$.
\end{corollary}

We give two examples of non-separating selfadjoint operator spaces, that generate a unital C*-algebra, but attain positive functionals that do not admit positive extensions.
In Example \ref{E:tables} we show that the conditions of approximate positive generation of $E$, unitality of $\ca(E)$, and the inclusion $E \subseteq \ca(E)$ being an embedding are independent.

\begin{example} \label{E:nopositive_ext1}
Let $\bT$ be the unit circle and consider the selfadjoint operator space 
\[
F := \{f \in {\rm C}(\bT): \int_{\bT} fdm=0 \}=\{f \in {\rm C}(\bT): \wh{f}(0)=0\},
\] 
where $m$ is the normalised Lebesgue measure and $\wh{f}(0)$ is the zero-th Fourier coefficient.
The coordinate function $z$ is in $F$, and thus we have $\ca(F)={\rm C}(\bT)$.
Moreover, $F$ is not separating for ${\rm C}(\bT)$ since the unital positive map 
\[
{\rm C}(\bT) \to \bC; f \mapsto \int_{\bT}f dm
\]
annihilates $F$. 
Note also that since any positive and non-zero continuous function has a non-zero integral, we have $F_+=\{0\}$. 

Consider the ``length" functional $\vphi\colon F \to \bC$ defined by
\[
 \vphi(f) = \int_{\bT} fd(\de_{1}-\de_{-1}) = f(1)-f(-1).
\]
Then $\vphi$ is bounded and it is trivially positive since $ F_{+}=\{0\}$. 
We claim that $\vphi$ does not extend to a positive functional on ${\rm C}(\bT)$, and thus the inclusion $F \subseteq {\rm C}(\bT)$ is not an embedding.

Suppose towards a contradiction that there exists a positive functional $\wt{\vphi} \colon {\rm C}(\bT) \to \bC$ extending $\vphi$.
Then $\wt{\vphi}$ is induced by a positive measure $\mu$ by the Riesz--Markov--Kakutani Theorem, and in particular
\[
\int_{\bT}f d\mu = \wt{\vphi}(f)  = \vphi(f)  = \int_{\bT} fd(\de_{1}-\de_{-1}) \foral f\in F.
\]
Let $g\in {\rm C}(\bT)$ and set $f:= g - (\int_{\bT} gdm) \cdot 1$; then $f \in F$ and thus
\[
\int_{\bT} g d\mu - \mu(\bT) \int_{\bT} g dm
=
\int_{\bT} f d\mu
=
\int_{\bT} fd(\de_{1}-\de_{-1}) 
=
\int_{\bT} g d(\de_1 - \de_{-1}).
\]
Again by the Riesz--Markov--Kakutani Theorem, we get
\[
\mu - \mu(\bT) \cdot m = \de_{1} - \de_{-1},
\]
and thus $\mu(\{-1\}) = -1 <0$.
This is a contradiction as $\mu$ is positive.
\end{example}

The problem in Example \ref{E:nopositive_ext1} is not the triviality of the cones as the following example indicates.

\begin{example}\label{E:nopositive_ext2}
Let $E \subseteq M_2(\bC)$ be the selfadjoint operator space generated by $\{E_{11}, E_{12}\}$.
Then $\ca(E) = M_2(\bC)$ as $E_{11}, E_{12}, E_{21} \in E$ and $E_{22} = E_{21} E_{12}$. The compression to the (2,2) entry is a state on $\ca(E)$ that annihilates $E$, and hence $E$ is not separating for $\ca(E)$.
Note also that $E$ is not approximately positively generated, as $E_{+} = \bR_{+} \cdot E_{11}$.
Let $\vphi \colon E \to \bC$ be the functional defined by
\[
\vphi(x) = {\rm Tr}( \begin{bmatrix} 0 & 1 \\ -1 & 0 \end{bmatrix} x)
\foral x \in E.
\]
This is a bounded functional with $\vphi(E_{11}) = 0$ and $\vphi(E_{12}) = -1$.
Hence it is (completely) positive on $E$.
We claim that $\vphi$ has no (completely) positive extension on $\ca(E) = M_2(\bC)$.

Suppose towards a contradiction that $\vphi$ extends to a positive functional $\wt{\vphi} \colon M_2(\bC) \to \bC$, i.e.,
\[
\wt{\vphi}(\cdot) = {\rm Tr}(\varrho \cdot)
\text{ for some }
\varrho \in M_2(\bC)_+.
\]
We then have $\varrho_{11} = {\rm Tr}(\varrho E_{11}) = \vphi(E_{11}) = 0$, and therefore $\varrho_{12} = \varrho_{21} = 0$ as well.
With that in hand we have
\[
-1 = \vphi(E_{12}) = {\rm Tr}(\varrho E_{22}) = 0,
\]
which is a contradiction.
\end{example}

In Example \ref{E:nonemb} we further show that approximate positive generation of $E$ on its own still does not guarantee that the inclusion in $\ca(E)$ is an embedding.
The problem is that positive extensions (which exist by Corollary \ref{C:exten}) may not preserve the norm.

On the other hand, a class of selfadjoint operator spaces where positive extensions to $\ca(E)$ with preservation of norm exist appears in \cite{DKP25}.

\begin{example}
Let $E \subseteq \ca(E)$ be a selfadjoint operator space. 
Suppose there is a C*-algebra $\A \subseteq E$ such that $\A \cdot E \subseteq E$, and an approximate unit $(e_\la)_\la$ for $\A$ such that $\lim_\la e_\la x = x$ for every $x\in E$. 
In \cite[Proposition 2.5]{DKP25} it is shown that every contractive  positive functional on $E$ extends to a  contractive  positive functional on $\ca(E)$.
Hence $E \subseteq \ca(E)$ is an embedding.

Such examples arise naturally in the context of \cite{CS21, CS22}.
Is $\S$ is an operator system and $\K$ is the compact operators, then $\S \otimes \K$ admits such an approximate unit.
A further class is formed by the selfadjoint operator spaces generated by a non-degenerate C*-correspondence.
\end{example}

Another class of selfadjoint operator spaces that allows positive extensions with preservation of the norm to the C*-algebra they generate arises from matrix ordered operator spaces with an approximate order unit that defines the matrix norms considered in \cite[Section 2.2]{CS22}.
We give the relevant definitions. Let $E$ be a matrix ordered operator space. An \emph{approximate order unit} for $E$ is a net $(e_\la)_\la\subseteq E_+$ that satisfies the following:
\begin{enumerate}
\item $e_\la\leq e_\mu$ if $\la\leq \mu$.
\item For every $x\in E_{sa}$ there exist a positive $t\in \bR$ and a $\la$ such that $-te_\la \leq x \leq t e_\la$.
\end{enumerate}

In \cite[Lemma 2.9]{CS22} it is shown that $e_\la^n:=(e_\la\otimes I_n)_\la$ is an approximate order unit for $M_n(E)$. 
We say that $(e_\la)_{\la}$ is \emph{matrix norm-defining} if for every $n\in \bN$ and $x\in M_n(E)$ we have
\[
\|x\|_{M_n(E)}=\inf\left\{t: \begin{bmatrix} te_\la^n & x \\ x^* & t e_\la^n \end{bmatrix}\in M_{2n}(E)_+\text{ for some } \la\right\}.
\]
By \cite[Proposition 2.10]{CS22} those spaces are selfadjoint operator spaces in our sense.

In order to obtain extensions, the authors in \cite{CS22} consider selfadjoint operator spaces $E \subseteq \ca(E)$ that contain a matrix norm-defining order unit for the \emph{entire} $\ca(E)$.
In the third paragraph of the proof of \cite[Proposition 2.13]{CS22}, it is used that, in this case, the inclusion of $E \subseteq \ca(E)$ extends to a unital complete order embedding $E^\# \subseteq \ca(E)^\#$, thus allowing the extension of a contractive positive functional.
In order to show that $E \subseteq \ca(E)$ is an embedding we can show that \cite[Lemma 2.11]{CS22} is an “if and only if” statement.

\begin{lemma}\label{L:orderunit}
Let $E$ be a selfadjoint operator space and let $(e_\la)_\la$ be matrix norm-defining approximate order unit for $E$. Let $\vphi \colon E \to \bC$ be a selfadjoint functional. 
Then, $\vphi$ is positive if and only if  $\| \vphi\| = \lim \vphi(e_{\la})$.

\end{lemma} 

\begin{proof}
If $\vphi$ is positive, then it follows from \cite[Lemma 2.11]{CS22} that its norm equals the limit of $ \vphi(e_\la)$. 
Conversely, let $x$ be in $E_{+}$ with $ \|x\|\leq 1$ and $\eps>0$ be arbitrary. 
Since the approximate order unit is matrix norm-defining, by \cite[Proposition 1.8(i)]{KV97} there exists a positive $t'\in \bR $ and a $\la$ such that 
\[
-t'e_{\la} \leq x\leq t' e_{\la}
\qand
t' < \|x\| + \eps \leq 1 +\eps.
\] 
Then for any $\mu \geq \la$ we have 
\[ 
-(1+ \eps)e_{\mu} \leq - t' e_{\la} \leq -x \leq e_{\mu }-x \leq e_{\mu} \leq (1+\eps)e_{\mu}.
\] 
Since the approximate order unit is contractive, we get
\[
| \vphi(x) - \vphi(e_{\mu}) |
= 
| \vphi(x- e_{\mu})| 
\leq 
\|\vphi\| \cdot \|x - e_\mu\|
\leq
(1 + \eps)\|\vphi \|.
\]
Taking the limit and using the assumption we obtain
\begin{align*}
| \vphi(x) - \| \vphi \| | \leq(1+\eps) \| \vphi \|.
\end{align*}
Since $\eps >0$ was arbitrary, we have
\begin{align*}
| \vphi(x) - \| \vphi \| | \leq \| \vphi \| 
\Longrightarrow 
- \| \vphi \| \leq \vphi(x) - \|\vphi \| \leq \| \vphi \|, 
\end{align*}
where we used that $\vphi$ is selfadjoint and $ \vphi(x) \in \bR$.
This implies that $ \vphi (x) \geq 0$.
\end{proof}

To complete the discussion in the context of \cite[Section 2.2]{CS22}, we now obtain the following proposition.
It will be superseded by Corollary \ref{C:mndaou}.

\begin{proposition}\label{P:mndaou}
Let $E \subseteq \ca(E)$ be a selfadjoint operator space.
If $E$ contains a matrix norm-defining approximate order unit $(e_\la)_\la$ for  $\ca(E)$, then $E\subseteq\ca(E)$ is an embedding.
\end{proposition}

\begin{proof}
By Theorem \ref{T:equivemb} it suffices to show that if $\vphi \colon M_n(E) \to \bC$ is a contractive positive map then it admits a contractive positive extension $\wt{\vphi}\colon M_n(\ca(E))\to \bC$.
By the Hahn--Banach Theorem, $\vphi$ can be extended to a functional $\wt{\vphi} \colon M_n(\ca(E)) \to \bC$ with $ \| \wt{\vphi} \| = \| \vphi \|$, and by replacing $\wt{\vphi}$ with $(\wt{\vphi} + \wt{\vphi}^{*})/2$ we may assume that $\wt{\vphi}$ is selfadjoint. 
Lemma \ref{L:orderunit} yields $\|\vphi\| = \lim \vphi(e_{\la}^n)$, and as $\|\wt{\vphi} \| = \|\vphi\|$, using Lemma \ref{L:orderunit} once more we obtain that $\wt{\vphi}$ is positive. 
\end{proof}

We next consider embeddings not just with respect to some generated C*-algebra but with respect to any completely isometric complete order embedding in a selfadjoint operator space.

\begin{definition}\label{D:aep}
We say that a selfadjoint operator space $E$ has \emph{the automatic embedding property} if every completely isometric complete order embedding $\vthe \colon E \to F$ with $F$ a selfadjoint operator space is an embedding.
\end{definition}

In particular we just need to check the completely isometric complete order Hilbertian embeddings.

\begin{lemma}\label{L:red}
Let $E$ be a selfadjoint operator space.
The following are equivalent:
\begin{enumerate}
\item $E$ has the automatic embedding property.
\item Every completely isometric complete order embedding $\vthe \colon E \to \B(H)$ is an embedding.
\end{enumerate}
\end{lemma}

\begin{proof}
That item (i) implies (ii) is trivial.
For the converse, let $\vthe \colon E \to F$ be a completely isometric complete order embedding.
Let $\vthe' \colon F \to \B(H)$ be an embedding, so that the unitisation $(\vthe')^{\#}$ is completely isometric.
Since $\vthe' \circ \vthe$ is automatically an embedding, the unitisation
\[
(\vthe' \circ \vthe)^{\#} = (\vthe')^{\#} \circ \vthe^{\#}
\]
is completely isometric.
Hence $\vthe^{\#}$ is completely isometric, and thus $\vthe$ is an embedding.
\end{proof}

We will prove that approximately positively generated selfadjoint operator spaces, where the positive decompositions are compatible with the norms, have the automatic embedding property.
Following the terminology used in \cite{HKM23} we give the next definitions.

\begin{definition}
Let $E$ be a selfadjoint operator space and  set 
\[
\left(M_n(E)_+\right)_{\ka}:=\{p\in M_n(E)_+ : \|p\|\leq \ka\}.
\]
\begin{enumerate}
\item We say that $E$ is \emph{completely $\ka$-generated} if there exists a $\ka\geq 1$ such that
\[
\left(M_n(E)_{sa}\right)_1 \subseteq  \left(M_n(E)_+\right)_{\ka} - \left(M_n(E)_+\right)_{\ka}.
\]

\item We say that $E$ is \emph{completely approximately $\ka$-generated} if there exists a $\ka\geq 1$ such that
\[
\left(M_n(E)_{sa}\right)_1 \subseteq \ol{\left(M_n(E)_+\right)_{\ka} - \left(M_n(E)_+\right)_{\ka}}.
\]
\end{enumerate}
\end{definition}

We will need the following lemma.

\begin{lemma} \label{L:qsposcon}
Let $E \subseteq \B(H)$ be a selfadjoint operator space and let $\vphi \in {\rm CCP}(E, \bC)$.
If $\vphi$ satisfies
\[
\|\vphi\| = \sup\{ \vphi(x) : x \in (E_+)_1\},
\]
then it extends to a $\wt{\vphi} \in {\rm CCP}(\B(H), \bC)$ (and consequently with $\nor{\wt{\vphi}} = \nor{\vphi}$).

In particular, every selfadjoint extension $\wt{\vphi}$ of $\vphi$ on $\B(H)$ with $\nor{\wt{\vphi}} = \nor{\vphi}$ is positive.
\end{lemma}

\begin{proof}
By the Hahn--Banach Theorem, let $\wt{\vphi}$ be a functional on $\B(H)$ with $\|\wt{\vphi}\| = \|\vphi\|$ that extends $\vphi$.
By replacing $\wt{\vphi}$ with $(\wt{\vphi} + \wt{\vphi}^*)/2$ we may suppose that $\wt{\vphi}$ is a selfadjoint extension of $\vphi$ with the same norm. 
Consider the Jordan decomposition
\[
\wt{\vphi} = \wt{\vphi}_+ - \wt{\vphi}_{-}
\text{ such that }
\|\wt{\vphi}\| = \|\wt{\vphi}_{+}\| + \|\wt{\vphi}_{-}\|,
\]
into positive functionals $\wt{\vphi}_{\pm}$ of $\B(H)$.
For $x \in E$ with $0 \leq x \leq I_H$ we have
\[
\vphi(x)
=
\wt{\vphi}(x)
=
\wt{\vphi}_{+}(x) - \wt{\vphi}_{-}(x)
\leq
\wt{\vphi}_{+}(x)
\leq
\wt{\vphi}_{+}(I_H)
=
\|\wt{\vphi}_+\|.
\]
Taking supremum over all such $x$ and using the assumption we get
\[
\|\wt{\vphi}_{+}\| + \|\wt{\vphi}_{-}\|
=
\|\wt{\vphi}\|
=
\|\vphi\|
\leq 
\|\wt{\vphi}_{+}\|.
\]
We conclude that $\|\wt{\vphi}_{-}\| \leq 0$ and so $\wt{\vphi} = \wt{\vphi}_{+}$.
\end{proof}

Consequently, we establish a sufficient condition under which a selfadjoint operator space admits the automatic embedding property.

\begin{corollary}\label{C:decomp}
If $E$ is a selfadjoint operator space such that 
\[
\|\vphi\| = \sup\{ \vphi(x) : x \in (M_n(E)_+)_1\}
\foral
\vphi \in {\rm CCP}(M_n(E), \bC), n \in \bN,
\]
then $E$ has the automatic embedding property.
\end{corollary}

\begin{proof}
Without loss of generality, and by Lemma \ref{L:red}, suppose that $E \subseteq \B(H)$ and we will show that this inclusion is an embedding. 
By Lemma \ref{L:qsposcon} any contractive positive functional of $M_n(E)$ extends to a contractive positive functional of $M_n(\B(H))$, and hence
Theorem \ref{T:equivemb} completes the proof.
\end{proof}

The condition of Corollary \ref{C:decomp} is satisfied when $E$ is a completely approximately 1-generated selfadjoint operator space.

\begin{theorem}\label{T:1apppos}
If $E$ is a completely approximately 1-generated selfadjoint operator space, then $E$ has the automatic embedding property.
\end{theorem}

\begin{proof}
Take $\vphi \in {\rm CCP}(M_n(E), \bC)$ and $\eps>0$. 
Since $\vphi$ is selfadjoint, we have 
\[
\|\vphi\|=\sup \{ \vphi(x) : x\in (M_n(E)_{sa})_1 \}.
\]
Moreover, since $E$ is completely approximately 1-generated, we obtain 
\[
\sup \{ \vphi(x) : x\in (M_n(E)_{sa})_1\}
=
\sup \{ \vphi(x) : x\in \left((M_n(E)_+)_1-(M_n(E)_+)_1\right) \cap (M_n(E)_{sa})_1  \}.
\]
Thus, we may pick an element $x \in (M_n(E)_{sa})_1$ such that $\vphi(x) > \|\vphi\|-\eps$, and $x = x_1 - x_2$ with $x_1, x_2 \in M_n(E)_+$  satisfying $\|x_1\|, \|x_2\| \leq 1$.
We then have
\[
\vphi(x_1)\geq \vphi(x_1)-\vphi(x_2)=\vphi(x)> \|\vphi\|-\eps.
\]
Since $\eps$ was arbitrary and $\|x_1\| \leq 1$ we obtain 
\[
\sup \{ \vphi(x') : x' \in (M_n(E)_+)_1\} = \|\vphi\|.
\]
Since $\vphi$ was arbitrary, Corollary \ref{C:decomp} completes the proof.
\end{proof}

As a further application, we improve Proposition \ref{P:mndaou} and require that the matrix norm-defining approximate order unit works just for the selfadjoint operator space, as opposed for the entire generated C*-algebra.

\begin{corollary}\label{C:mndaou}
Let $E$ be a selfadjoint operator space.
If $E$ contains a matrix norm-defining approximate order unit $(e_\la)_\la$ for $E$, then $E$ has the automatic embedding property.
\end{corollary}

\begin{proof}
It follows from Lemma \ref{L:orderunit} and Corollary \ref{C:decomp} by noting that $e_\la^n$ is contractive and positive for every $\la$ and $n\in \bN$.
\end{proof}

The following corollary is a refinement of \cite[Theorem 4.4 and Theorem 4.5]{Rus23}. 
In particular, we show that the inclusion of operator systems need not to be unital, and the inclusion of C*-algebras need not to be multiplicative, in order to have gauge maximal inclusions.

\begin{corollary}\label{C:emb}
The following hold:
\begin{enumerate}
\item If $\phi \colon \S \to \T$ is a completely isometric complete order embedding between operator systems, then $\phi$ is an embedding.
\item If $\phi \colon \C \to \D$ is a completely isometric complete order embedding between C*-algebras, then $\phi$ is an embedding.
\end{enumerate}
\end{corollary}

\begin{proof}
Both cases follow from Theorem \ref{T:1apppos} since operator systems and C*-algebras are completely 1-generated selfadjoint operator spaces.
\end{proof}

\begin{remark}\label{R:Arvext}
If $\S \subseteq \T$ are operator systems with $1_\S\neq 1_\T$, then Corollary \ref{C:emb} and Theorem \ref{T:arvemb} yield that every completely contractive completely positive map $\phi \colon \S \to \B(K)$ extends to a completely contractive completely positive map $\phi \colon \T\to \B(K)$.
Thus Arveson's Extension Theorem holds even for non-unital inclusions of operator systems. 
(We note that completely contractive completely positive extensions of unital completely positive maps on operator systems are automatically unital.)
\end{remark}

The following remark was communicated to us by Russell.

\begin{remark} \label{R:unique}
By Corollary \ref{C:emb} we have that if $\S$ is an operator subsystem in an operator system $\T$ (not necessarily with the same units) with co-dimension $1$, then $\T$ is completely order isomorphic to $E^{\#}$.
\end{remark}

We finish the section with an example to demonstrate that the assumption of \emph{complete (approximate) 1-generation} in Theorem \ref{T:1apppos} cannot be relaxed to mere \emph{(approximate) positive generation} if one wishes to obtain an automatic embedding property.

\begin{example}\label{E:nonemb}
Let $E \subseteq M_2(\bC)$ be the selfadjoint operator space generated by
\[
A = \begin{bmatrix} 
1 & 1 \\ 
1 & 1 
\end{bmatrix} \qand 
B = 
\begin{bmatrix}
0 & 0 \\ 
0 & 1 
\end{bmatrix}.
\]
Since $\{A, B\}$ is a linearly independent selfadjoint set, any selfadjoint element in $E$ has the form
\[
X = \al A + \be B = 
\begin{bmatrix} 
\al & \al 
\\ \al & \al + \be 
\end{bmatrix}, \quad 
\text{where } \al, \be \in \bR.
\]
Moreover, the positive cone of $E$ is
\[
E_+ = \{ \al A + \be B : \al, \be \geq 0 \}.
\]
Let $X = \al A + \be B \in E_{sa}$ be arbitrary. 
By setting
\[
\al_1 = \max(\al, 0),\quad \al_2 = \max(-\al, 0) \qand
\be_1 = \max(\be, 0), \quad \be_2 = \max(-\be, 0),
\]
we obtain 
\[
X = (\al_1 A + \be_1 B) - (\al_2 A + \be_2 B)\in E_+-E_+,
\]
and hence $E$ is positively generated.

\medskip

\noindent
{\bf Claim 1.} The space $E$ is not completely 1-generated.

\noindent
{\it Proof of Claim 1.} Towards this end, let
\[
X = A - B = 
\begin{bmatrix} 
1 & 1 \\ 
1 & 0 
\end{bmatrix} 
\in E.
\]
We will show that for any decomposition $X=X_1- X_2$ with $X_1, X_2\in E_+$ we have 
\[
\max\{\|X_1\|, \|X_2\|\} > \|X\| 
=  \frac{1 + \sqrt{5}}{2}.
\]

Indeed, suppose that $X = X_1 - X_2$ with $X_1, X_2 \in E_+$ such that $X_1 = \al_1 A + \be_1 B$ and $X_2 = \al_2 A + \be_2 B$. 
Comparing entries we obtain $\al_1 = 1 + \al_2$ and $\be_1 = \be_2 - 1$.
Hence, by setting $\al_2 := \al$ and $\be_2 := \be$ we have the decomposition 
$X = X_1(\al, \be) - X_2(\al, \be)$
with
\[
X_1(\al, \be) = 
\begin{bmatrix} 
1+ \al & 1 + \al \\ 
1 + \al & \al + \be 
\end{bmatrix}
\qand
X_2(\al, \be) = 
\begin{bmatrix} 
\al & \al \\ 
\al & \al + \be 
\end{bmatrix}.
\]
By positivity we get
\[
X_2(\al, \be) \geq 0 \Rightarrow \al \geq 0, \be \geq 0
\qand
X_1(\al, \be) \geq 0 \Rightarrow \frac{\al + \be}{1+\al} \geq 1 \Rightarrow \be \geq 1.
\]
The eigenvalues of $X_1(\al,\be)$ are
\[
\la_1 = \frac{1}{2}\left(-\sqrt{4\al^2 + 8\al + \be^2 - 2\be +5} + 2\al + \be + 1\right)
\]
and
\[
\la_2 = \frac{1}{2}\left(\sqrt{4\al^2 + 8\al + \be^2 - 2\be +5} + 2\al + \be + 1\right).
\]
Since $\al\geq 0$ and $\be \geq 1$ we get
\[
\|X_1(\al, \be)\|
=
\max\{ |\la_1|, |\la_2| \}
\geq
\la_2
\geq
2
=
\|X_1(0,1)\|.
\]
Therefore we have
\[
\max\{ \|X_1(\al, \be)\|, \|X_2(\al, \be)\|\geq
\|X_1(\al,\be)\| 
\geq 2 > \|X\|,
\]
as required. \hfill{$\Box$}

\medskip

\noindent
{\bf Claim 2.} The inclusion $ E \subseteq \ca(E)= M_2(\bC)$ is not an embedding. 

\noindent
{\it Proof of Claim 2.}
The equality $\ca(E)=M_2(\bC)$ follows directly from the computations 
\[
E_{11} = A - AB-BA+B , \quad
E_{12} = AB-B, \quad
E_{21} = BA-B \qand
E_{22} = B.
\]
By Theorem \ref{T:equivemb} it suffices to show that some $\vphi \in {\rm CCP}(E, \bC)$ does not extend to a $\wt{\vphi} \in {\rm CCP}(M_2(\bC), \bC)$ with the same norm. 
Towards this end, let  $\vphi \colon E \to \bC$ given by 
\[
\vphi (\al A + \be B) := \be/\sqrt{2}. 
\]
Then $ \vphi $ is clearly positive and we will prove that $\|\vphi\|=1$. 
Choosing
$\al = -\sqrt{2}^{-1}$ and $\be = \sqrt{2}$,
we get an element
\[
X = 
\begin{bmatrix} 
-\frac{1}{\sqrt{2}} & -\frac{1}{\sqrt{2}} \\[6pt] 
-\frac{1}{\sqrt{2}} & \frac{1}{\sqrt{2}} 
\end{bmatrix}
\in E,
\]
with  $\|X\| = 1$ and $\vphi (X) = 1$; hence $\|\vphi \| \geq 1$.
Let $(e_i)_{i=1}^2 \subseteq \bC^2$ denote the canonical basis of $\bC^2$, and consider the unit vector
$z := \sqrt{2}^{-1}(e_1- e_2)\in \bC^2$. Then for $X=\al A+ \be B\in E$ we have
\[
X z = \al
\begin{bmatrix} 
1 & 1 \\ 
1 & 1 
\end{bmatrix} 
\frac{1}{\sqrt{2}} 
\begin{bmatrix} 
1 \\
-1 
\end{bmatrix} 
+ \be
\begin{bmatrix} 
0 & 0 \\ 
0 & 1 
\end{bmatrix} 
\frac{1}{\sqrt{2}} 
\begin{bmatrix} 
1 \\ 
-1 
\end{bmatrix} 
=
-\frac{\be}{\sqrt{2}} e_2,
\]
and hence
$\|X\| \geq \frac{|\be|}{\sqrt{2}}=|\vphi(X)|$.
This concludes that $\|\vphi\| = 1$.

Now assume towards a contradiction that there exists a state $\wt{\vphi} \colon M_2(\bC) \to \bC$ that extends $ \vphi$ with the same norm. 
Pick $\varrho \in M_2(\bC)_+$ with ${\rm Tr}(\varrho )=1$ such that $\wt{\vphi} (\cdot ) = {\rm Tr}(\varrho \cdot)$. 
Then 
\[ 
0 = \vphi(A) = \wt{\vphi}(A) = \rm Tr(\varrho A)= \varrho_{11} + \varrho_{22} + 2 \rm Re(  \varrho_{12}) =1 + 2 \rm Re(  \varrho_{12}) 
\Rightarrow
\rm{Re}(\varrho_{12}) = -1/2
\]
On the other hand,
\[
\sqrt{2}^{-1} = \vphi(B) = \wt{\vphi}(B)  =\rm{Tr}(\varrho B) = \varrho_{22}.
\]
As $\varrho\in M_2(\bC)_+$ we have 
\[
\left(1 - \frac{1}{\sqrt{2}}\right) \cdot \frac{1}{\sqrt{2}} = \varrho_{11} \varrho_{22} \geq |\varrho_{12}|^2 \geq \left(\operatorname{Re}(\varrho_{12})\right)^2  = \frac{1}{4},
\]
and thus we arrive to the contradiction $4\sqrt{2} \geq 6$. \hfill{$\Box$}

\smallskip

Note that $\vphi$ does extend to a positive functional on all of $M_2(\bC)$ by Corollary \ref{C:exten}, but not with preservation of norm. 
One such extension is the functional $\wt{\vphi} \colon M_2(\bC) \to \bC$ given by 
\[
\wt{\vphi}(C) := \sca{ C v,v} \text{ for } v = \sqrt[4]{2}^{-1} (e_1-e_2) \in \bC^2.
\]
\end{example}

\section{One-dimensional selfadjoint spaces}\label{S:1d}

In this section we consider one dimensional selfadjoint subspaces of $\B(H)$.
Recall that C*-inclusions preserve the spectrum and the decomposition of selfadjoint elements into the positive and negative parts by functional calculus.
We begin with the following observation.

\begin{lemma}\label{L:1dcon}
Let $\C$ be a C*-algebra and let $E = \spn\{x\}$ be the selfadjoint subspace generated by an element $x \in \C_{sa}$.
Then the following hold:
\begin{enumerate}
\item If $x_{\pm} \neq 0$, then $M_n(E) \cap M_n(\C)_+ = \{0\}$.
\item If $x_{+} = 0$ or $x_{-} = 0$, then for every $y \in M_n(E)_{sa}$ there are $y_1, y_2 \in M_n(E)_+$ such that $y = y_1 - y_2$ and $\max\{\|y_1\|, \|y_2\|\} = \|y\|$. In particular, $E$ is completely $1$-generated.
\end{enumerate}
\end{lemma}

\begin{proof}
(i) Suppose that $a \otimes x \geq 0$ in $M_n(\C)$, and so $(a \otimes x)_{-} = 0$.
Therefore we have 
\[
a_{+} \otimes x_{-} = 0 
\qand
a_{-} \otimes x_{+} = 0.
\]
Since $x_{\pm} \neq 0$, we get $a_{\pm} = 0$.
Hence $a=0$, and thus $a \otimes x = 0$.

\smallskip

\noindent
(ii) Without loss of generality assume that $x_{-} = 0$, since $\spn\{x\} = \spn\{-x\}$.
Suppose that $a \otimes x \in M_n(E)_{sa}$.
Since $x = x_+ \in E_+$ and $a = a^*$ we get
\[
(a \otimes x)_+ = a_{+} \otimes x \in M_n(E)_+
\qand
(a \otimes x)_{-} = a_{-} \otimes x \in M_n(E)_+.
\]
The norm equality follows from the properties of the decomposition.
\end{proof}

As a consequence of Theorem \ref{T:1apppos} and Lemma \ref{L:1dcon} we obtain the following.

\begin{corollary}\label{C:1dproper}
Let $\C$ be a C*-algebra and let $E= \spn\{x\}$ be the selfadjoint subspace generated by an element $x \in \C_{sa}$.
If $x_{+} = 0$ or $x_{-} = 0$, then the inclusion $E \subseteq \C$ is an embedding.
\end{corollary}

We next prove that inclusions of one-dimensional generated selfadjoint spaces allow for completely positive extensions (that may not preserve the norm). 
We first need the following remark.

\begin{remark}\label{R:autbound}
Let $x$ be in $\B(H)$ and let $E=\spn\{x\}$. 
Then every  linear map $\phi \colon E \to \B(K)$ is automatically completely bounded with $\|\phi\|_{\rm cb} = \|\phi\| =\|x\|^{-1} \|\phi(x)\|$.
For the remainder of this section, we will not use the term (completely) bounded when we refer to linear maps defined on one-dimensional operator spaces.
\end{remark}

\begin{proposition}\label{P:1dposext}
Let $x\in \B(H)$ be a selfadjoint operator and set $E := \spn\{x\}$. 
Then, every completely positive map $\phi \colon E \to \B(K)$ extends to a completely positive map $\wt{\phi} \colon \ca(E) \to \B(K)$, for every $n \in \bN$. 
In particular, the following hold: 
\begin{enumerate}
\item If $ x_{+}=0$ or $ x_-=0$, then there is such an extension $\wt{\phi}$ with
\[
\|\wt{\phi}\|_{\rm cb}= \|\phi\|_{\rm cb}.
\]
\item If $ x_{\pm}\neq 0$,  then there is such an extension $\wt{\phi}$ with
\[
\|\phi\|_{\rm cb} \;\le\; \|\wt{\phi}\|_{\rm cb} \;\le\; \frac{\|x\|}{\min\{\|x_+\|,\|x_-\|\}} \, \|\phi\|_{\rm cb}.
\]
\end{enumerate}
\end{proposition}

\begin{proof}
Item (i) follows from Theorem \ref{T:arvemb} and Corollary \ref{C:1dproper}.
We now prove item (ii), so suppose that $ x_{\pm}\neq 0$.
Let $\phi\colon E \to \B(K)$ be a completely positive map. Denote by $\si(x)$ the spectrum of $x$ relatively to $\B(H)$, and set 
\[
m := \min \si(x) \qand M := \max \si(x).
\]
Since $ x_{\pm}\neq 0$ we have $m < 0 < M$.
Set
\[
T := \phi(x) \in B(K),
\]
and let $T = T_+ - T_-$ be its decomposition with $T_\pm \ge 0$ and $T_-T_+=T_+ T_- = 0$.  
Define
\[
Y := M^{-1} T_+ - m^{-1} T_- \in B(K)_+.
\]

Now let $\fB(\si(x))$ denote the Borel subsets of $\si(x)$ and define the map $Q \colon \fB(\si(x)) \to \B(K)$ such that
\[
Q(\{M\}) := \frac{T - mY}{M-m} \geq 0
\qand
Q(\{m\}) := \frac{MY - T}{M-m} \geq 0,
\]
and $Q(B)=0$ for every $B\in \fB(\si(x))$ with $B\cap \{m,M\}=\mt$.
Then $Q$ is a bounded regular and positive $\B(K)$-valued measure satisfying
\[
\mu_{\xi,\eta}
= \langle Q(\{m\})\xi,\eta\rangle\, \de_{m}
+ \langle Q(\{M\})\xi,\eta\rangle\, \de_{M}.
\]

Define the (completely) positive map  
\[
\Phi\colon {\rm C}(\si(x)) \to B(K); f \mapsto  \int_{\si(x)} f(t)\, dQ(t) = f(m) Q(\{m\}) + f(M) Q(\{M\})
\]
(see for example \cite[Theorem 3.11 and Proposition 4.5]{Pau02}).
Equivalently,
\[
\langle \Phi(f)\xi,\eta\rangle = \int f \, d\mu_{\xi,\eta} \foral \xi,\eta \in K.
\]
For the coordinate function $z$ we have 
\[
\Phi(z)
= m\,Q(\{m\}) + M\,Q(\{M\}) 
= m\,\frac{MY-T}{M-m} + M\,\frac{T-mY}{M-m}  
= T  =\phi(x).
\]
Hence $\Phi|_E=\phi$ after identifying $\ca(x,I_H)$ with ${\rm C}(\si(x))$. 

We now compute the norms. First note that  
\[
\Phi(1)=Q(\{m\})+ Q(\{M\})= \frac{T - mY}{M-m} + \frac{MY - T}{M-m}= Y,
\]
and thus $\|\Phi\|_{\rm cb} = \|\Phi(1)\| = \|Y\|$.
Since $T_+$ and $T_-$ have orthogonal ranges, we get
\[
\|\Phi\|_{\rm cb} = \|Y\| = \max\Big\{ \tfrac{\|T_+\|}{M}, \tfrac{\|T_-\|}{|m|} \Big\} \leq \frac{\|T\|}{\min\{M,|m|\}} \leq \frac{\|x\|}{\min\{M,|m|\}} \|\phi\|_{\rm cb}.
\]
By setting $\wt{\phi}:= \Phi|_{\ca(E)}$ we obtain a completely positive extension of $\phi$ such that
\[
\|\phi\|_{\rm cb}  \leq \|\wt{\phi}\|_{\rm cb} \leq \|\Phi\|_{\rm cb} \leq \frac{\|x\|}{\min\{\|x_+\|,\|x_-\|\}} \|\phi\|_{\rm cb},
\]
thus showing item (ii).
\end{proof}

\begin{remark}
With the setup of Proposition \ref{P:1dposext}, Remark \ref{R:Arv_cor} yields that, for every $n\in \bN$, every positive functional $\vphi\colon M_n(E) \to \bC$ admits a positive extension $\wt{\vphi}\colon M_n(\ca(E)) \to \bC$. 
However, the positive extensions do not always preserve the norm.

For such an example, let $x$ be a selfadjoint element in a C*-algebra $\C$ such that 
\[
0 < \min\{ \|x_{+}\|, \|x_{-}\| \} < \|x\|;
\]
equivalently, $x_{\pm} \neq 0$ and either $\|x\|\notin \si_{\C}(x)$ or $-\|x\|\notin \si_{\C}(x)$. 
Set $E := \spn\{x\} \subseteq \C$. 
Then
\[
\dist(x, E_+) = \| x\| = \dist(-x, E_+),
\]
as $E_{+}=\{0\}$.
By Proposition \ref{P:gaugeC*} we also have
\[
\dist(x, \C_+) = \| x_{+} \|
\qand
\dist(-x, \C_+) = \| x_{-} \|.
\]
Since $\min\{\|x_+\|,\|x_-\|\}<\|x\|$, we obtain that $E\subseteq \C$ is not a gauge maximal inclusion, which by Theorem \ref{T:equivemb} implies that there exists a $n\in \bN$ and a $\vphi \in {\rm CCP}(M_n(E), \bC)$ that does not extend to a completely contractive completely positive map $\wt{\vphi}\colon \ca(M_n(E)) \to \bC$.
\end{remark}

For the remainder of the section we focus in the case where $x \in \B(H)_{sa}$ with $x_{\pm} \neq 0$.
The only obstruction to having an embedding of $\spn\{x\} \subseteq \B(H)$ in this case, is the possible complete positive extension (of either and thus of both) of the maps
\[
\phi \colon \spn\{x\} \to \spn\{x\}; \phi(x) = -x
\qand
\vphi \colon \spn\{x\} \to \bC; \vphi(x) = -\|x\|.
\]
In view of item (i) of Lemma \ref{L:1dcon} these maps are completely contractive completely positive on $\spn\{x\}$.

\begin{theorem} \label{T:1d}
Let $x\in \B(H)$ be a selfadjoint operator such that $x_{\pm} \neq 0$ and set $E := \spn\{x\}$.
Then the following are equivalent:
\begin{enumerate}
\item The inclusion of $\spn\{x\}$ inside $\ca(E)$ is an embedding.
\item $\|x_{+}\| = \|x_{-}\|$.
\item The map $\phi \colon \spn\{x\} \to \spn\{x\}\subseteq \B(H)$ with $\phi(x) = -x$ has a completely contractive completely positive extension $\wt{\phi} \colon \ca(E) \to \B(H)$.
\item The functional $\vphi \colon \spn\{x\} \to \bC$ with $\vphi(x) = -\|x\|$ has a contractive positive extension $\wt{\vphi} \colon \ca(E) \to \bC$.
\end{enumerate}
\end{theorem}

\begin{proof}
Without loss of generality let us assume that $\|x\| = \|x_{+}\| = 1$ by normalising, and using the fact that $\|x\| = \max\{ \|x_{+}\|, \|x_{-}\| \}$; otherwise consider $-x$ in the place of $x$.

\smallskip

\noindent
[(i) $\Rightarrow$ (ii)].
Suppose that the inclusion $E \subseteq \ca(E)$ is an embedding,.
Then we have
\[
\|x_{+}\| = \|x\| = \dist(x, E_+) = \dist(x, \ca(E)_+) = \|x_{-}\|
\]
where we used that $E_+ = \{0\}$ in the second equality, the equivalence of items (i) and (iii) of Theorem \ref{T:equivemb} in the third equality, and Proposition \ref{P:gaugeC*} in the last equality.

\smallskip

\noindent
[(ii) $\Rightarrow$ (iii)]. It follows from Proposition \ref{P:1dposext} (ii). 

\smallskip

\noindent
[(iii) $\Rightarrow$ (iv)]. Since 
\[
\|x\|=\|x_+\|\in \si_{\ca(E)}(x)
\]
by the functional calculus we may pick a state $\tau \colon \ca(E) \to \bC$ such that $\tau(x)= \|x\|$. 
By assumption $\phi$ has a completely contractive completely positive extension $\wt{\phi}$, and we set $\wt{\vphi}: =\tau \circ \wt{\phi}$ to obtain a contractive positive extension of $ \vphi$.

\smallskip

\noindent
[(iv) $\Rightarrow$ (i)]. In view of Theorem \ref{T:arvemb} we will show that any  completely contractive completely positive map $\psi \colon E \to M_n(\bC)$ has a completely contractive completely positive extension to $\ca(E)$.
Since $x$ is selfadjoint, then $\psi(x) \in M_n(\bC)_{sa}$ and there is a unitary $u \in M_n(\bC)$ such that
\[
u \psi(x) u^* = \diag\{d_1, \dots, d_n\}
\text{ such that }
d_i \in \bR \text{ and } |d_i| \leq 1 \foral i=1, \dots, n.
\]
By Lemma \ref{L:1dcon} we have $E_+=\{0\}$, and hence the map
\[
\psi_i \colon E \to \bC; \psi_i(x) := d_i
\qfor
i=1, \dots, n,
\]
is contractive and positive.
We will show that for every $i=1,\dots, n$ there exists a contractive positive map $\wt{\psi}_i \colon \ca(E) \to \bC$ that extends $\psi_i$.
With that in hand the map
\[
\wt{\psi} := \ad_{u^*} \circ (\oplus_{i=1}^n \wt{\psi}_i) \colon \ca(E) \to M_n(\bC)
\]
is completely contractive and completely positive on $\ca(E)$, and it extends $\psi$ since
\[
\wt{\psi}(x) = u^* \diag\{d_1, \dots d_n\} u = \psi(x).
\]

If $d_i \geq 0$, then choose a state $\tau$ of $\ca(E)$ such that $\tau(x) = \|x\| = 1$, and set $\wt{\psi}_i := d_i \tau$, which is contractive and positive on $\ca(E)$.
If $d_i \leq 0$, then set $\wt{\psi}_i := -d_i \wt{\vphi}$, which is contractive and positive on $\ca(E)$ by the assumption on the existence of $\wt{\vphi}$.
In both cases we have that $\wt{\psi}_i$ is an extension of $\psi_i$, $i=1, \dots, n$.
\end{proof}

\begin{example} \label{E:exemb} \cite[Example 2.14]{KKM23}
For $r\in (0,1]$, consider the space $E_r \subseteq M_2(\bC)$ generated by the matrix
\[
x_r = \begin{bmatrix} 1 & 0 \\ 0 & -r \end{bmatrix}.
\]
Then by Theorem \ref{T:1d} the inclusion $E_1 \subseteq M_2(\bC)$ is an embedding. 
Note that the map $\phi \colon x_1 \mapsto -x_1$ is implemented by the change of basis $e_1 \leftrightarrow e_2$, and the map $\vphi \colon x_1 \mapsto -1$ is implemented by the compression to the (2,2) entry.
\end{example}

We close with the following promised examples that shows that approximate positive generation of $E$, unitality of $\ca(E)$, and the embedding property of $E \subseteq \ca(E)$ are independent.

\begin{example}\label{E:tables}
Let us first present existence of examples in the following cases when $E \subseteq \ca(E)$ is an embedding.

\vspace{2pt}

\begin{center}
\begin{tabular}{| c | c | c |}
\hline
& \; $E$ is approx. pos. gen. \; & \; $E$ is not approx. pos. gen. \; \\ \hline
\; $\ca(E)$ is unital \; & (i) & (ii) \\  \hline
\; $\ca(E)$ is not unital \; & (iii) & (iv)  \\ \hline
\end{tabular}

\vspace{2pt}

Table 1. When $E \subseteq \ca(E)$ is an embedding.
\end{center}

\smallskip

\noindent
(i) Consider $E$ to be any operator system.

\smallskip

\noindent
(ii) Consider the selfadjoint operator spaces 
\[
E:= \{ f\in {\rm C}(\bT): f(z)= \al z + \be\bar z , \; \al, \be \in \bC \} \qand F := \{f \in {\rm C}(\bT): \int_{\bT} fdm=0 \}
\]  
of Examples \ref{E:torus} and \ref{E:nopositive_ext1}. 
Then $\ca(E) = \ca(F) = {\rm C}(\bT)$, and as shown therein the inclusion $E\subseteq {\rm C}(\bT)$ is an embedding, and $E$ is not separating for ${\rm C}(\bT)$ as $F\supseteq E$ is not separating for ${\rm C}(\bT)$. 
Hence, by Proposition \ref{P:sepequiv} we obtain that $E$ is also not approximately positively generated.  

\smallskip

\noindent
(iii) Consider $E = \S \otimes \K$ for an operator system $\S$ and the compact operators $\K \subseteq \B(\ell^2(\bN))$. 
Then $\ca(E)=\ca(\S)\otimes \K$, and by \cite[Proposition 2.5]{DKP25} we have that $E \subseteq \ca(\S)\otimes \K$ is an embedding.

\smallskip

\noindent
(iv) Consider $E \subseteq \K \subseteq \B(\ell^2(\bN)$ be the selfadjoint
operator space generated by the element
\[
x = x_+ - x_{-}
\, \text{ for } \,
x_+ = \sum_{n=0}^\infty \frac{1}{n + 1} p_{2n}
\, \text{ and } \,
x_{-} = \sum_{n=0}^\infty \frac{1}{n + 1} p_{2n +1}.
\]
Then $\ca(E)=c_0$ and since $x$ is a non-positive selfadjoint element with $\|x_+\| = \|x_{-}\|$, Theorem \ref{T:1d} implies that $E \subseteq c_0$ is an embedding.

\smallskip

We next present existence of examples in the following cases when $E \subseteq \ca(E)$ is not an embedding.

\vspace{2pt}

\begin{center}
\begin{tabular}{| c | c | c |}
\hline
& \; $E$ is approx. pos. gen. \; & \; $E$ is not approx. pos. gen. \; \\ \hline
\; $\ca(E)$ is unital \; & (v) & (vi) \\  \hline
\; $\ca(E)$ is not unital \; & (vii) & (viii)  \\ \hline
\end{tabular}

\vspace{2pt}

Table 2. When $E \subseteq \ca(E)$ is not an embedding.
\end{center}

\smallskip

\noindent
(v) The space $E$ of Example \ref{E:nonemb} shows that $E$ is positively generated, $\ca(E)=M_2(\bC)$ and thus unital, and $E \subseteq M_2(\bC)$ is not an embedding.

\smallskip

\noindent
(vi) Consider $E \subseteq M_2(\bC)$ be the selfadjoint operator space
generated by the element
\[
x = \begin{bmatrix} 1 & 0 \\ 0 & -1/2 \end{bmatrix}.
\]
Then $\ca(E)=\bC\oplus \bC$ and as shown in \cite[Example 2.14]{KKM23} and  \cite[Example 4.1]{Rus23}, the inclusion $E \subseteq \bC\oplus \bC$ is not an embedding.

\smallskip

\noindent
(vii) Let $E_1 \subseteq \B(H_1)$ be approximately positively generated, with $\ca(E_1)$ unital, and assume that the inclusion $E_1 \subseteq \ca(E_1)$ is not an embedding, for example as in case (v).
Let $E_2 \subseteq \B(H_2)$ be approximately positively generated, with $\ca(E_2)$ non-unital, for example as in case (iii).
Then the space $E = E_1 \oplus E_2$ is approximately positively generated, $\ca(E) = \ca(E_1) \oplus \ca(E_2)$ is not unital, and the inclusion $E \subseteq \ca(E)$ is not an embedding.

\smallskip

\noindent
(viii) Consider $E \subseteq \K \subseteq \B(\ell^2(\bN))$ be the selfadjoint operator space generated by the element
\[
x = x_+ - x_{-}
\, \text{ for } \,
x_+ = \sum_{n=0}^\infty \frac{1}{n + 1} p_{2n}
\, \text{ and } \,
x_{-} = \sum_{n=0}^\infty \frac{1}{n + 2} p_{2n +1}.
\]
Since $x$ is a non-positive selfadjoint element with $\|x_+\| \neq \|x_{-}\|$, Theorem \ref{T:1d} implies that $E \subseteq \ca(E)$ is not an embedding.
\end{example}

\section{Co-universality in different classes} \label{S:ext}

\subsection{Terminal objects}

We now  inspect the different classes of morphisms in the category of selfadjoint operator spaces and their terminal objects.
Let $E$ be selfadjoint operator space and let $\F$ be a subclass of the class
\[
\F_1 := \{\phi \colon E\to \B(K) : \text{$\phi$ completely isometric completely positive} \}.
\]
We say that \emph{$\F$ has a terminal object} if there exists a $j \in \F$ such that for every $\phi \in \F$ there exists a $*$-epimorphism $\Phi \colon \ca(\phi(E)) \to \ca(j(E))$ with $\Phi(\phi(x)) = j(x)$ for every $x\in E$.
If such a $j$ exists then $\ca(j(E))$ is unique up to canonical $*$-isomorphisms (i.e., $*$-isomorphisms that fix the copies of $E$). We will write $\partial \F$ for $\ca(j(E))$.
We consider the following subclasses of $\F_1$:
\begin{align*}
\F_2 & := \{\vthe \colon E\to \B(K) : \text{$\vthe$ completely isometric complete order embedding} \}; \\
\F_3 &:= \{\vthe \colon E \to \B(K) : \text{$\vthe$ embedding} \}.
\end{align*}

By definition $\F_3\subseteq \F_2 \subseteq \F_1$, and by \cite[Theorem 2.25]{CS21} we get that $(\partial \F_3, j_3)$ always exists and $\partial \F_3=\cenv(E)$. 
Blecher--Kirkpatrick--Neal--Werner show in \cite[Theorem 2.3]{BKNW07} that $(\partial \F_1, j_1)$ exists when $E$ is approximately positively generated. 
However in general $(\partial F_1, j_1)$ or $(\partial F_2, j_2)$ may not exist.
This is the case for $(\partial F_2, j_2)$ in Example \ref{E:F_2} (that has trivial positive cones), and the same arguments rule out the existence of $(\partial \F_1, j_1)$ as well.
For $\F_2$ we have the following proposition.

\begin{proposition}\label{P:clas}
Let $E$ be a selfadjoint operator space. 
Then the following are equivalent:
\begin{enumerate}
\item $\F_2 = \F_3$.
\item $\partial \F_2$ exists and there is a canonical $*$-isomorphism $\partial \F_2\cong \partial \F_3$.
\end{enumerate}
\end{proposition}

\begin{proof}
If $\F_2=\F_3$, then it is immediate that $\partial \F_2=\partial \F_3$. 
For the converse, let $\Phi \colon \partial \F_2 \to \partial \F_3$ be a $*$-isomorphism such that $\Phi(j_2(x))=j_3(x)$ for every $x\in E$.
Suppose that $\vthe \in \F_2$ and we will show that $\vthe \in \F_3$. 
By the definition of $(\partial \F_2, j_2)$ let $\pi \colon \ca(\vthe(E)) \to \partial \F_2$ be a $*$-epimorphism such that $\pi(\vthe(x))=j_2(x)$ for every $x\in E$. 
For any $x\in M_n(E)_{sa}$ we have 
\begin{align*}
\dist(j_3^{(n)}(x), - M_n(\partial \F_3)_+)
& =
\inf\{\|x+p\| \colon p\in M_n(E)_+\} \\
& = 
\inf\{\|\vthe^{(n)}(x)+\vthe^{(n)}(p)\| \colon p\in M_n(E)_+\}\\
& \geq 
\dist(\vthe^{(n)}(x), - M_n(\ca(\vthe(E)))_+),
\end{align*}
where in the first equality we use Theorem \ref{T:equivemb} and that $j_3$ is an embedding. 
Moreover, we have
\begin{align*}
\dist(j_3^{(n)}(x), - M_n(\partial \F_3)_+)
& =
\inf\{\|(\Phi\circ\pi)^{(n)}(\vthe^{(n)}(x))+p\| \colon p\in M_n(\partial \F_3)_+\}\\
& \leq
\inf\{\|(\Phi\circ\pi)^{(n)}(\vthe^{(n)}(x))+(\Phi\circ\pi)^{(n)}(p)\| \colon p\in M_n(\ca(\vthe(E)))_+\}\\
& \leq
\dist(\vthe^{(n)}(x), - M_n(\ca(\vthe(E)))_+).
\end{align*}
Hence $\vthe$ is a gauge maximal isometry, and by Theorem \ref{T:equivemb} it is an embedding.
\end{proof}

Although approximate positive generation implies the existence of $\partial F_1$, it does not imply a canonical identification with $\partial F_3$.

\begin{example}
Let $E\subseteq M_2(\bC)$ be the selfadjoint operator space of Example \ref{E:nonemb}. 
By \cite[Theorem 2.3]{BKNW07} we have that $\partial \F_1$ exists since $E$ is positively generated. 
We will prove that there is no canonical $*$-isomorphism between $\partial \F_1$ and $\partial \F_3$.

Indeed, suppose that there is such a $*$-isomorphism, say $\Phi \colon \partial \F_1 \to \partial \F_3$ and let $\vphi\colon E\to \bC$ be the contractive and positive functional defined in Example \ref{E:nonemb} that does not admit a contractive positive extension on $M_2(\bC)$.
By the defining property of $\partial \F_1$ there exists a $*$-epimorphism $\pi \colon M_2(\bC) \to \partial \F_1$ such that $\pi(x)=j_1(x)$ for all $x\in E$.
The map $j_3\colon E\to \partial \F_3$ is an embedding and hence by Theorem \ref{T:equivemb} we may pick a contractive and positive map $\psi \colon \partial \F_3\to \bC$ that extends $\vphi \circ j_3^{-1}$. 
Setting $\wt{\vphi}:=\psi\circ \Phi \circ \pi$ gives a contractive and positive extension of $\vphi$ on $M_2(\bC)$, which is a contradiction.

Note also that even if $\partial \F_2$ exists then there is no canonical $*$-isomorphism between $\partial \F_2$ and $\partial \F_3$. 
This follows by Proposition \ref{P:clas} as in this case the inclusion $E\subseteq \ca(E)$ would be in $ \F_2$ and not in $\F_3$.
\end{example}

As a corollary of Theorem \ref{T:1apppos} we obtain the following.

\begin{corollary}\label{C:F_2}
Let $E$ be a completely approximately 1-generated selfadjoint operator space.
Then $\partial \F_2$ exists and $\partial \F_2 = \partial \F_3$.
\end{corollary}

In particular for operator systems we have the following corollary.

\begin{corollary}
Let $\S$ be an operator system. 
Then $\partial \F_1 \cong \partial \F_2 \cong \partial \F_3\cong \cenv(\S)$ via canonical $*$-isomorphisms.
\end{corollary}

\begin{proof}
Since $\S$ is completely 1-generated, the result follows from Corollary \ref{C:F_2} combined with \cite[Theorem 2.25 (ii)]{CS21} and the comments following \cite[Theorem 2.3]{BKNW07}.
\end{proof}

\subsection{Hyperrigidity}

The hyperrigidity property was introduced by Arveson \cite{Arv11} for a generating set $\G$ in a C*-algebra $\ca(\G)$.
A set $\G$ is called \emph{hyperrigid} if for every faithful non-degenerate $*$-representation $\Phi \colon \ca(\G) \to \B(K)$ and every sequence of unital completely positive maps $\phi_n \colon \B(K) \to \B(K)$ the following holds:
\[
\lim_n \|\phi_n(\Phi(g)) - \Phi(g)\|=0 \foral g \in \G
\Longrightarrow
\lim_n \|\phi_n(\Phi(c)) - \Phi(c)\| = 0 \foral c \in \ca(\G).
\]

Arveson \cite{Arv11} shows that if an operator system $\S$ is hyperrigid inside the C*-algebra $\ca(\S)$ it generates, then $\ca(\S)$ is canonically $*$-isomorphic with $\cenv(\S)$. 
It appears that having an embedding, together with either the separating property or the unitality of the generated C*-algebra, is sufficient.

\begin{proposition} \label{P:embd_plus_hyper}
Let $E\subseteq \ca(E)$ be a selfadjoint operator space that is hyperrigid in $\ca(E)$.
Suppose that the inclusion $E\subseteq \ca(E)$ is an embedding. Then $\cenv(E)\cong \ca(E)$ by a canonical $*$-isomorphism if at least one of the following conditions is satisfied:
\begin{enumerate}
\item $E$ is separating for $\ca(E)$.
\item $\ca(E)$ is unital.
\end{enumerate}
\end{proposition}

\begin{proof}
Let $(\cenv(E),j)$ denote the C*-envelope of $E$. 
Since the inclusion $E\subseteq \ca(E)$ is an embedding, by the defining property of $\cenv(E)$ we obtain a $*$-epimorphism $\Phi \colon \ca(E)\to \cenv(E)$ satisfying $\Phi(x)=j(x)$ for every $x\in E$.

\smallskip

\noindent
(i) Assume that $E$ is separating for $\ca(E)$.
Since $E\subseteq \ca(E)$ is an embedding, we may identify $E^\#$ with $E+ \bC 1_{\ca(E)^\#}$ inside $\ca(E)^\#$. 
Since $E$ is separating for $\ca(E)$, by \cite[Proposition 3.20]{DKP25} we have that $E^\#$ is hyperrigid in $\ca(E)^\#$. 
Consequently, we obtain a canonical $*$-isomorphism $\cenv(E^\#)\cong \ca(E)^\#$. 
The proof of \cite[Theorem 2.25]{CS21} implies that $\cenv(E)$ is $*$-isomorphic to the C*-algebra generated by the copy of $E$ inside $\cenv(E^\#)$, and hence the injective $*$-homomorphism
\[
\Psi \colon \cenv(E) \hookrightarrow \cenv(E^\#)\cong \ca(E)^\#
\]
satisfies $\Psi(j(x))=(x,0)$ for every $x\in E$, which implies that the  range of $\Psi$ is $\ca(E)$. 
We conclude that $\Psi$ and $\Phi$ are mutual inverses and thus $\Phi$ is a $*$-isomorphism.

\smallskip

\noindent
(ii) Assume that $\ca(E)$ is unital.  Let $ \pi \colon \ca(E) \to \B(K)$ be a faithful unital $*$-representation.
Since $E$ is hyperrigid, then $\pi$ has the unique extension property by  \cite[Theorem 3.10, Theorem 3.11]{DKP25}. 
Consider the completely contractive completely positive map $\pi \circ j^{-1} \colon j(E) \to \B(K)$. 
By Theorem \ref{T:arvemb} we may extend it to a completely contractive completely positive map
\[
\phi\colon \cenv(E) \to \B(K)  \text{ so that } \phi \circ j = \pi|_E. 
\]
On the other hand $\phi \circ j = \phi \circ \Phi|_E$, and by the unique extension property of $\pi $ we obtain $\phi \circ \Phi = \pi$.
Since $ \pi $ is injective we conclude that $\Phi$ is also injective, and thus a $*$-isomorphism. 
\end{proof}

Consequently we have the following corollary.

\begin{corollary} \label{C:emb_apr_hyper}
Let $E\subseteq \ca(E)$ be a selfadjoint operator space that is hyperrigid in $\ca(E)$.
Then $\cenv(E)\cong \ca(E)$ by a canonical $*$-isomorphism in the following cases:
\begin{enumerate}
\item $\A \subseteq E$ is a C*-algebra such that $\A \cdot E \subseteq E$ and there is an approximate unit $(e_\la)_\la \subseteq \A$ such that $\lim_\la e_\la x = x$ for all $x \in E$.
\item $E$ contains a matrix norm-defining approximate order unit $(e_\la)_\la$ for $E$.
\item $E$ is completely approximately 1-generated.
\end{enumerate}
\end{corollary}

\begin{proof}
The fact that the inclusion $E\subseteq \ca(E)$ is an embedding in item (i), follows from \cite[Proposition 2.5]{DKP25}.
By the same result we have that $E$ is separating for $\ca(E)$.
The fact that $E$ is separating for $\ca(E)$ in item (ii) follows in the same way by Lemma \ref{L:orderunit}, and the embedding property follows from Corollary \ref{C:mndaou}.
Item (iii) follows from Proposition \ref{P:sepequiv} and Theorem \ref{T:1apppos}.
\end{proof}

We will say that a selfadjoint operator space $E\subseteq \ca(E)$ \emph{contains enough unitaries in} $\ca(E)$ if there is a collection of unitaries in $E$ which generates $\ca(E)$ as a C*-algebra, that is, $\ca(E)$ is the smallest C*-algebra that contains these unitaries. 
The notion of a subspace containing enough unitaries was first considered by Pisier in \cite{Pis96} in the context of operator spaces to provide a simple proof of Kirchberg's Theorem. 
It was then studied by Kavruk--Paulsen--Todorov--Tomforde in the context of nuclearity and tensor products of operator systems \cite{KPTT13}.

\begin{proposition}\label{P:unitaruep}
Let $E\subseteq \ca(E)$ be  a selfadjoint operator space that contains enough unitaries in $ \ca(E)$. If $\Phi \colon \ca(E) \to \B(K)$ is a unital $*$-representation then $\Phi|_E$ has the unique extension property. In particular, $E$ is hyperrigid.
\end{proposition}

\begin{proof}
It follows using similar arguments as in \cite[Lemma 9.3]{KPTT13} and using \cite[Theorem 3.10]{DKP25}.
\end{proof}

As a consequence of Proposition \ref{P:embd_plus_hyper} and Proposition \ref{P:unitaruep} we obtain the following result for the C*-envelope of a selfadjoint operator space that contains enough unitaries in the C*-algebra it generates.

\begin{corollary}\label{C:un and emb}
Let $ E\subseteq \ca(E)$ be a selfadjoint operator space that contains enough unitaries in $\ca(E)$. 
Then the following are equivalent:
\begin{enumerate}
\item The inclusion $E\subseteq \ca(E)$ is an embedding.
\item $\cenv(E)\cong \ca(E)$ by a canonical $*$-isomorphism.
\end{enumerate}
\end{corollary}

\begin{remark}
We note that having a selfadjoint operator space that contains enough unitaries in the C*-algebra it generates does not imply that the inclusion is an embedding. 
Indeed, consider the selfadjoint operator space $F \subseteq {\rm C}(\bT)$ of Example \ref{E:nopositive_ext1}. 
Then the inclusion in ${\rm C}(\bT)$ is not an embedding, but $F$ contains enough unitaries in ${\rm C}(\bT)$ as it contains the coordinate function. In particular, by Proposition \ref{P:unitaruep} we have that $F$ is hyperrigid in ${\rm C}(\bT)$ and yet Corollary \ref{C:un and emb} implies that there is no canonical $*$-isomorphism between ${\rm C}(\bT)$ and $\cenv(F)$.
\end{remark}

\begin{example}
Recall the selfadjoint operator space $E_{n} \subseteq \ca(\bF_{n})$ of Example \ref{E:torus}. 
As shown therein, this inclusion is an embedding, and since $E_{n}$ contains enough unitaries in $\ca(\bF_{n})$, by Corollary \ref{C:un and emb} we have a canonical $*$-isomorphism $\cenv(E_{n})\cong \ca(\bF_{n})$.

We note also that the selfadjoint operator space  $E_{n}$ satisfies a universal property analogous to the operator system case \cite[Proposition 9.7]{KPTT13}: for every selfadjoint operator space $F$ and contractions $ y_{1}, \dots,y_n \in F$, there exists a completely contractive completely positive map $ \phi\colon E_n \to F$ mapping $ \phi (u_i) =y_{i}$, for all $ i =1,\dots,n$. 
Indeed, consider an embedding $ F\subseteq \B(H)$ and as in Example \ref{E:torus}, dilate each $ y_i$ to a unitary $\wh{y_i}$ in $\B(H^{(2)})$. 
By the universal property of $\ca(\bF_{n})$, there exists a unital $*$-homomorphism 
\[
\Phi\colon \ca(\bF_{n}) \to \B(H^{(2)}) \text{ such that }
\Phi(u_i) =\wh{y_i} \foral i =1, \dots, n.
\]
Compressing to the $(1,1)$ corner and restricting to $E_{n}$ yields a completely contractive completely positive map $\phi\colon E_n \to F$ such that $ \phi(u_i)=y_i$, $ i =1, \dots, n$.
\end{example}

\subsection*{Acknowledgements}

The authors would like to thank Travis Russell for the helpful comments on a draft of this manuscript, that led to Corollary \ref{C:Wun=Run} and Remark \ref{R:unique}. 
The authors would like to thank the referee for the constructive comments and remarks that helped improve the submitted manuscript.
Alexandros Chatzinikolaou, Evgenios Kakariadis and Apollonas Paras\-kevas acknowledge that this research work was supported within the framework of the National Recovery and Resilience Plan Greece 2.0, funded by the European Union - NextGenerationEU (Implementation Body: HFRI. Project name: Noncommutative Analysis: Operator Systems and Nonlocality. HFRI Project Number: 015825).
Sam Kim acknowledges that this research work was supported by the Research Foundation Flanders (FWO) research project G085020N
and the internal KU Leuven funds project number C14/19/088.
Apollonas Paraskevas acknowledges that this research work was supported by the Hellenic Foundation for Research and Innovation (HFRI) under the 5th Call for HFRI PhD Fellowships (Fellowship Number: 19145).
Research work for this project was undertaken during the ``Operator Systems and Applications" meeting at Banff (February 2025), and the ``Noncommutative Analysis: Operator systems and nonlocality" Summer School at Samos (July 2025).
The authors would like to thank the Banff International Research Station and the University of Aegean for the hospitality.

\begin{open}
For the purpose of open access, the authors have applied a Creative Commons Attribution (CC BY) license to any Author Accepted Manuscript (AAM) version arising.
\end{open}


\end{document}